\PassOptionsToPackage{svgnames}{xcolor}

\documentclass[a4paper,11pt,fleqn,reqno,oneside]{amsart}

\usepackage{lmodern}

\usepackage{mathtools}

\usepackage{amssymb}
\usepackage{amsthm}
\usepackage{amsfonts}

\usepackage{mathrsfs} 
\usepackage[english]{babel}
\usepackage{upgreek}
\usepackage{esvect}
\usepackage{extarrows}
\usepackage[noadjust]{cite}

\usepackage[foot]{amsaddr}

\usepackage{comment}

\let\Gamma\varGamma

\usepackage{xcolor}

\usepackage[
	paper = a4paper,
	top = 32mm, bottom = 32mm,
	left = 34mm, right = 34mm,
	footskip = 10mm, footnotesep=5mm, 
]
{geometry}

\usepackage{doi}

\usepackage{hyperref}
\hypersetup{
  pdftitle    = {Evolution of Contractions between Non-Compact Manifolds under the Mean Curvature Flow},
  pdfsubject  = {},
  pdfauthor   = {Felix Lubbe},
  pdfkeywords = {},
  pdfcreator  = {},
  pdfproducer = {LaTeX with hyperref},
  hidelinks, 
  colorlinks,
  linkcolor=blue,
  citecolor=blue,
  urlcolor=RoyalBlue
}

\usepackage[sorted-cites]{amsrefs}

\usepackage{etoolbox}
\patchcmd{\section}{\scshape}{\bfseries}{}{}
\makeatletter
\renewcommand{\@secnumfont}{\bfseries}
\patchcmd{\@startsection}
  {\@afterindenttrue}
  {\@afterindentfalse}
  {}{}
\makeatother

\usepackage{enumitem}

\newcommand{\Nb}{\mathbb{N}}
\newcommand{\Rb}{\mathbb{R}}

\newcommand{\Hb}{\mathbb{H}}

\newcommand{\Sym}{\mathrm{Sym}}

\DeclareMathOperator*{\rank}{\mathrm{rank}}
\DeclareMathOperator{\tr}{\mathrm{tr}}

\newcommand{\Ac}{\mathcal{A}}
\newcommand{\Bc}{\mathcal{B}}
\newcommand{\Cc}{\mathcal{C}}
\newcommand{\Hc}{\mathcal{H}}

\newcommand{\ee}{\mathrm{e}}

\newcommand{\<}{\langle}
\renewcommand{\>}{\rangle}

\newcommand{\der}{\mathrm{d}}
\DeclareMathOperator{\dd}{d\hspace{-2.5pt}}

\newcommand{\dt}{\partial_t}

\newcommand{\Lap}{\Updelta}

\newcommand{\D}{\mathrm{D}\mkern 1mu}

\newcommand{\gMetric}{\mathrm{g}}
\newcommand{\gM}{\gMetric_M}
\newcommand{\gN}{\gMetric_N}
\newcommand{\gMN}{\gMetric_{M\times N}}
\newcommand{\gind}{\gMetric}
\newcommand{\gNorm}{\gMetric^{\perp}}

\newcommand{\h}{\mathrm{h}}

\newcommand{\tgind}{\widetilde{\gind}}

\newcommand{\sTensor}{\mathrm{s}}
\newcommand{\sMN}{\sTensor_{M\times N}}
\newcommand{\sind}{\sTensor}
\newcommand{\sN}{\sTensor^{\perp}}

\newcommand{\piM}{\pi_M}
\newcommand{\piN}{\pi_N}
\newcommand{\RpiM}{\pi_{\Rb^m}}

\newcommand{\Rm}{\mathrm{R}}
\newcommand{\Rind}{\Rm}

\newcommand{\RmN}{\Rm_N}
\newcommand{\RmMN}{\Rm_{\Rb^m\times N}}
\DeclareMathOperator{\Ric}{Ric}

\newcommand{\A}{\mathrm{A}}
\newcommand{\Hv}{\vv{H}} 			

\newcommand{\RgM}{\gMetric_{\Rb^m}}
\newcommand{\RgN}{\gMetric_N}
\newcommand{\RgMN}{\gMetric_{\Rb^m\times N}}
\newcommand{\RsMN}{\sTensor_{\Rb^m\times N}}

\newcommand{\nN}{\nabla^{\perp}}
\newcommand{\LapN}{\Updelta^{\perp}}
\newcommand{\prN}{\mathrm{pr}^{\perp}}

\newcommand{\sv}{\lambda} 

\DeclareSymbolFont{UPM}{U}{eur}{m}{n}
\DeclareMathSymbol{\partial}{0}{UPM}{"40}

\newcommand{\fScaled}[1]{\widetilde{f}_{#1}}
\newcommand{\tScaled}{r}
\newcommand{\xScaled}{y}

\usepackage{enumitem}

\usepackage{enumitem}

\newtheorem{mainthm}{Theorem}

\newtheorem{theorem}{Theorem}
\newtheorem{lemma}[theorem]{Lemma}
\newtheorem{corollary}[theorem]{Corollary}
\newtheorem{proposition}[theorem]{Proposition}

\theoremstyle{definition}
\newtheorem{remark}[theorem]{Remark}
\newtheorem*{remark*}{Remark}
\newtheorem{definition}[theorem]{Definition}
\newtheorem*{definition*}{Definition}
\newtheorem{example}[theorem]{Example}

\numberwithin{equation}{section}
\numberwithin{theorem}{section}

\allowdisplaybreaks

\title[Contractions between Non-Compact Manifolds]{Evolution of Contractions\\between Non-Compact Manifolds}
\author{Felix Lubbe}
\date{\today}
\keywords{mean curvature flow, length-decreasing maps, non-compact manifold}
\subjclass[2010]{Primary 53C44; 53C42, 53C21}

\begin{document}

\begin{abstract}
	Let $N$ be a complete manifold with bounded geometry, such that
	$\sec_N\le -\sigma < 0$ for some positive constant $\sigma$.
	We investigate the mean curvature flow of the graphs
	of smooth length-decreasing maps $f:\mathbb{R}^m\to N$.
	In this case, the solution exists for all times and the evolving submanifold
	stays the graph of a length-decreasing map $f_t$.
	We further prove uniform decay estimates for all derivatives
	of order $\ge 2$ of $f_t$ along the flow.
\end{abstract}

\maketitle
\tableofcontents

\section{Introduction}

\setlength{\mathindent}{1.5cm}

Let $(M,\gM)$ and $(N,\gN)$ be complete Riemannian manifolds, and consider a smooth map $f:M\to N$.
The map $f$ is called \emph{strictly length-decreasing}
or \emph{contraction}, if there exists a fixed $\delta\in(0,1]$, 
such that
\begin{equation}
	\label{eq:LengthDecr}
	\|\dd f(v)\|_{\gN} \le (1-\delta)\|v\|_{\gM}
\end{equation}
holds for all $v\in\Gamma(TM)$.
In the present paper, we want to deform the map $f$ by deforming its graph
\[
	\Gamma(f) \coloneqq \big\{ (x,f(x)) \in M\times N : x\in M \big\}
\]
via the mean curvature flow inside the product space $M\times N$. That is, we
consider the system
\[
	\partial_t F_t(x) = \Hv(x,t) \,, \qquad F_0(M) = \bigl( x, f(x) \bigr) \,,
\]
where $\Hv(x,t)$ denotes the mean curvature vector of the submanifold $F_t(M)$ in $M\times N$
at $F_t(x)$.
A smooth solution to the mean curvature flow for which $F_t(M)$ is a graph for $0<T_g\le \infty$
can be described completely in terms of a smooth
family of maps $f_t:M\to N$, with $f_0=f$.
In the case of long-time existence of the graphical solution (i.\,e.\ $T_g=\infty$) and convergence, we would
thus obtain a smooth homotopy from $f$ to a minimal map $f_{\infty}:M\to N$. \\

In the compact case, there are many results for length- and area-decreasing maps
(see e.\,g.\ \cite{LL11,SHS13,SHS14,SHS15,Smo04,TW04,Wan01a,Wan01b} and references therein). 
For example, if $f:M\to N$ is strictly area-decreasing, $M$ and $N$ are space forms with $\dim M\ge 2$
subject to the relations
\[
	\sec_M \ge |\sec_N|\,, \qquad \sec_M+\sec_N > 0 \,,
\]
Wang and Tsui proved long-time existence of the graphical mean curvature flow and convergence
of $f_t$ to a constant map \cite{TW04}. The curvature assumptions were then weakened by
Lee and Lee \cite{LL11} and Savas-Halilaj and Smoczyk \cite{SHS14}.

In the non-compact case, Ecker and Huisken considered the flow of hypersurfaces in $\Rb^{n+1}$ and entire graphs with
at most linear growth and provided conditions under
which the initial graphs asymptotically approach self-expanding solutions \cite{EH89}. 
In fact, the growth assumption for the long-time existence theorem can be removed, 
so that only a Lipschitz condition on the initial graph is required \cite{EH91}.
In the higher-codimensional setting, however, due to the complexity of the normal bundle of the graph,
the methods of Ecker and Huisken cannot be applied.

Nevertheless, by considering the Gau{\ss} map of the immersion, several results in this setting were obtained
(see e.\,g.\ \cite{Wan03b,Wan05}).
When considering two-dimensional graphs, Chen, Li and Tian
established long-time existence and convergence results by evaluating angle 
functions on the tangent bundle \cite{CLT02}.
Further, there are some results 
showing long-time existence and convergence of the flow under smallness conditions
on the differential of the defining map \cite{CCH12,CCY13,SHS15}.

For Lagrangian graphs $\Gamma(f)\subset\Rb^m\times\Rb^m$ 
generated by Lipschitz continuous functions $f:\Rb^m\to\Rb^m$,
Chau, Chen and He showed
short-time existence of solutions with bounded geometry,
as well as decay estimates for the mean curvature vector and all higher-order derivatives
of the defining map, which imply the long-time existence of the solution \cite{CCH12}. 
This result was generalized in \cite{CCY13} by relaxing the length-decreasing condition
and subsequently by the author to strictly length-decreasing maps between Euclidean spaces
of arbitrary dimension \cite{Lub16}. A similar theorem also holds if one considers
the mean curvature flow of strictly area-decreasing maps between two-dimensional
Euclidean spaces \cite{Lub18}. \\

The aim of the present article is to prove estimates and long-time existence for strictly length-decreasing
maps  $f:\Rb^m\to N$, where $(N,\gN)$ is a complete Riemannian manifold with \emph{bounded geometry},
i.\,e.\ for any integer $k\ge 0$ we have
\[
	\sup_{x\in N} \|\nabla^k \RmN(x)\| < \infty
\]
and the injectivity radius satisfies $\mathrm{inj}(N)>0$. Similarly, 
a map $f:M\to N$ between two complete Riemannian manifolds $(M,\gM)$ and $(N,\gN)$ 
has \emph{bounded geometry}, if it satisfies
\[
	\sup_{x\in M} \|\nabla^k \dd f(x) \| < \infty \quad \text{for all} \quad k\ge 0 \,.
\]

Let us remark that the length-decreasing condition \eqref{eq:LengthDecr} is 
essentially measured by the difference 
\[
	\sind \coloneqq \RgM - f^*\gN
\]
of the two metrics $\RgM$ and $\gN$. 
The estimate for the eigenvalues of $\sind$ in the following theorem is given in terms
of an average, as given by
\[
	\tr(\sind) = \sum_{i,j=1}^m \gind^{ij} \sind_{ij} \,,
\]
where $\gind$ is the induced metric on the graph $\Gamma(f)\subset \Rb^m\times N$.

\begin{mainthm}
	\label{thm:ThmA}
	Let $(N,\gN)$ be a complete Riemannian
	manifold with bounded geometry.
	Assume that $N$ has negative sectional curvature, i.\,e.\ there is $\sigma>0$ with
	\[
		\sec_N \le -\sigma \,.
	\]
	Further, let $f:\Rb^m\to N$ be a smooth strictly length-decreasing map with bounded geometry.
	Then the mean curvature flow with initial condition $F_0(x)\coloneqq\bigl(x,f(x)\bigr)$
	has a long-time solution such that the following statements hold.
	\begin{enumerate}[label=(\roman*)]
		\item\label{mainthm:i} The evolving submanifold stays the graph of a strictly length-decreasing map
			$f_t:\Rb^m\to N$ for all $t>0$. 
		\item\label{mainthm:ii} The trace $\tr(\sind)$ is non-decreasing in time. If $m>1$ and $\inf_{\Rb^m\times\{0\}}\tr(\sind)<m-1$,
			then the estimate
			\[
				\tr(\sind) \ge \frac{C_1 (m-1) \exp\left(\frac{\sigma}{2}t\right) - m}{C_1\exp\left(\frac{\sigma}{2}t\right)-1}
			\]
			holds, where
			\[
				C_1 \coloneqq 1 + \frac{1}{(m-1)-\inf_{\Rb^m\times\{0\}}\tr(\sind)} \,.
			\]
		\item\label{mainthm:iii} The mean curvature vector of the graph stays bounded, i.\,e.\ there
			is a constant $C_2\ge 0$, such that
			\[
				\|\Hv\|^2 \le C_2 \,.
			\]
		\item \label{mainthm:iv} All spatial derivatives of order $k\ge 2$ of $f_t$ satisfy the estimate
			\[
				t^{k-1} \sup_{x\in\Rb^m} \bigl\| \nabla^{k-1} \dd f_t(x) \bigr\|^2 \le C_{k,\delta} \qquad \text{for all} \; k\ge 2
			\]
			and for some constants $C_{k,\delta}\ge 0$ depending only on $k$ and
			$\delta$.
	\end{enumerate}
\end{mainthm}

Let us shortly comment on the strategy of the proofs. As in \cite{Lub16,CCH12},
the idea is to identify suitable functions and symmetric tensors
to which a maximum principle can be applied. Since
Hamilton's tensorial maximum principle \cite{Ham86} is not applicable in the non-compact case, we
follow an idea from \cite{CCH12} in order to extend it to the setting at hand. \\

Before introducing the precise geometric setting for the proof, let us make
a few remarks on the assumptions and consequences of the theorem.

\begin{remark}[Long-time Behavior]
	Theorem \ref{thm:ThmA} implies that $f_t$ becomes stationary for $t\to\infty$.
	In particular, in this limit the map $f_t:(\Rb^m,\RgM)\to(N,\gN)$ becomes totally geodesic.
	
	However, depending on the initial conditions, the map $f_t$ may
	exhibit different long-time behavior.
	To see this, let us consider maps $f_t:\Rb\to\Hb^2$
	with initial datum $f_0$. Here, we use the disk model of $\Hb^2$ for
	illustration.
	\begin{enumerate}[label=(\roman*)]
		\item If $f_0$ is a geodesic (up to scaling), it is $f_t(x)=f_0(x)$ for all $t\ge 0$.
			In particular, the height of $f_t$ as a graph over $f_0$ is zero
			(see example \ref{ex:HS3}).
		\item If $f_0(\Rb)$ is a circle centered at the origin in $\Hb^2$, i.\,e.\
			$f_0(x)=r_0 \bigl(\sin(x),\cos(x)\bigr)$ for $0\le r_0<1$, 
			the image $f_0(\Rb)$ shrinks homothetically as a submanifold of $\Hb^2$. Thus,
			the height of $f_t$ as a graph over $f_0$ remains finite
			(see example \ref{ex:HS2}).
		\item If $f_0(\Rb)$ is a circle with one point at spatial infinity in $\Hb^2$,
			the image $f_t(\Rb)$ moves out to spatial infinity.
			We observe that the height of $f_t$ as a graph over $f_0$ is unbounded
			(see example \ref{ex:HS1}).
	\end{enumerate}
\end{remark}

\begin{remark}
	If $f_0$ maps into a compact region $K\subset N$,
	then the estimate \ref{mainthm:iii} on the mean curvature vector implies that $f_t$ also maps into
	compact regions $K_t$ for all finite times.
\end{remark}

\begin{remark}
	If $\sigma = 0$, \ref{mainthm:i}, \ref{mainthm:iii} and \ref{mainthm:iv}
	of theorem \ref{thm:ThmA} still hold. Thus, in this weaker formulation,
	the theorem can be applied e.\,g.\ to maps $f:\Rb^m \to N_1\times N_2$, where
	$N_1$ and $N_2$ have non-positive sectional curvatures. 
				
	For example, let $N$ be an arbitrary Riemannian manifold with non-positive
	sectional curvature, let $p_0\in N$ be fixed, $m\ge k$, and let
	$A:\Rb^m\to\Rb^k$ be a linear map satisfying 
	$\|Au\|^2_{\Rb^k} \le (1-\delta) \|u\|^2_{\Rb^m}$ for all $u\in\Rb^m$ and some $\delta\in(0,1]$.
	In particular, $A$ may be chosen such that
	the map $f:\Rb^m\to \Rb^k \times N$ given by $f(x) \coloneqq ( A x, p_0 )$
	satisfies $\tr(\sind) = c$ for any constant $c \in (m-k,m]$. Note that 
	the mean curvature flow with initial datum
	$F(x)\coloneqq \bigl(x, f(x) \bigr)$ is stationary.
\end{remark}

\subsection*{Acknowledgments} 
The author would like to thank Klaus Kr\"oncke,
Oliver Lindblad Petersen and David Lindemann for useful remarks and stimulating discussions.

\section{Maps between Euclidean Spaces}
\label{sec:Maps}

\subsection{Geometry of Graphs}
\label{sec:GeomOfGraphs}
We recall the geometric quantities in a graphical
setting, where we mostly follow the presentation in \cite{SHS15}*{Section 2}. \\

Let $(M,\gM)$ and $(N,\gN)$ be Riemannian manifolds
of dimensions $m$ and $n$, respectively.
On the product manifold $(M\times N,\gMN\coloneqq\gM\times\gN)$,
the two natural projections
\[
	\piM : M\times N \to M\,,\qquad \piN: M\times N\to N\,,
\]
are submersions, that is they are smooth and have maximal rank. 
A smooth map $f:M\to N$ defines an embedding $F:M\to M\times N$ via
\[
	F(x) \coloneqq \bigl(x,f(x)\bigr)\,,\qquad x\in M\,.
\]
The \emph{graph of $f$} is defined to be the submanifold 
\[
	\Gamma(f) \coloneqq F(M) = \left\{ \bigl(x,f(x)\bigr) : x\in M \right\} \subset M\times N \,.
\]
Since $F$ is an embedding, it induces another Riemannian metric on $M$, given by
\[
	\gind \coloneqq F^{*}\gMN \,.
\]
The four metrics $\gM,\gN,\gMN$ and $\gind$ are related by
\begin{align*}
	\gMN &= \piM^{*} \gM + \piN^{*}\gN \,, \\
	\gind &= F^{*}\gMN = \gM + f^{*}\gN \,.
\end{align*}
Further, $F$ defines an orthogonal splitting of the bundle
\[
	F^*T(M\times N) = \dd F(TM) \oplus T^{\perp}M \,,
\]
which induces a splitting of a vector field $v\in \Gamma\bigl(F^*T(M\times N)\bigr)$ as
\[
	v = v^{\top} \oplus v^{\perp} \,.
\]
We call $v^{\top}$ the \emph{tangential part of $v$} and $v^{\perp}$ the \emph{normal part of $v$}.
The projection onto the normal part is denoted by $\prN:F^*T(M\times N)\to T^{\perp}M$.
Using a local $\gind$-orthonormal frame $\{e_1,\dotsc,e_m\}$ of $TM$, $\prN$ can be expressed as
\[
	\prN(\xi) = \xi - \sum_{k=1}^m \gMN\bigl( \xi,\dd F(e_k) \bigr) \dd F(e_k) \,.
\]
As in \cites{SHS13,SHS14}, let us introduce the symmetric $2$-tensors
\begin{align*}
	\sMN &\coloneqq \piM^{*}\gM - \piN^{*}\gN\,, \\
	\sind &\coloneqq F^{*}\sMN = \gM - f^{*}\gN \,.
\end{align*}
Note that $\sMN$ is a semi-Riemannian metric of signature $(m,n)$ on the
manifold $M\times N$.
Like in \cite{SHS15}, let us also introduce
\[
	\sN(\xi,\eta) \coloneqq \sMN\bigl(\prN(\xi),\prN(\eta)\bigr) \,, \qquad \xi,\eta \in \Gamma \bigl( F^*T(M\times N) \bigr) \,.
\]
We denote the restriction of $\gMN$ to the normal bundle
by $\gNorm$.  \\

The Levi-Civita connection on $M$ with respect to the induced
metric $\gind$ is denoted by $\nabla$ and the corresponding curvature
tensor by $\Rind$. By restricting Levi-Civita connection $\nabla^{\gMN}$ of $M\times N$ 
to the normal bundle, 
we obtain the \emph{normal connection}, given by
\[
	\nN_v \xi \coloneqq \prN\bigl( \nabla^{\gMN}_{\dd F(v)} \xi \bigr) \,, \qquad v\in \Gamma(TM) \,, \quad \xi\in \Gamma\bigl(F^*T(M\times N)\bigr) \,.
\]

\subsection{Second Fundamental Form}

The \emph{second fundamental tensor} of the graph $\Gamma(f)$ is the section 
$\A\in\Gamma\bigl(T^{\perp}M \otimes \Sym(T^*M\otimes T^*M)\bigr)$
defined as
\[
	\A(v,w) \coloneqq \bigl( \nabla \dd F\bigr)(v,w) \coloneqq \nabla^{\gMN}_{\dd F(v)} \dd F(w) - \dd F(\nabla_v w) \,,
\]
where $v,w\in \Gamma(TM)$ and where we denote the connection on 
$F^{*}T(M\times N)\otimes T^{*}M$ induced by the Levi-Civita connection also by $\nabla$.
The trace of $\A$ with respect to the metric $\gind$ is called the
\emph{mean curvature vector field} of $\Gamma(f)$ and it will be denoted by
\[
	\Hv \coloneqq \tr \A \,.
\]
Let us denote the evaluation of the second fundamental
form (resp.\ mean curvature vector) in the direction of a vector $\xi\in\Gamma\bigl(F^*T(M\times N)\bigr)$ by
\[
	\A_{\xi}(v,w) \coloneqq \gMN\bigl( \A(v,w),\xi \bigr) \qquad \text{resp.} \qquad \Hv_{\xi} \coloneqq \gMN\bigl(\Hv, \xi\bigr) \,.
\]
Note that $\Hv$ is a section in the normal bundle of the graph.
If $\Hv$ vanishes identically, the graph is said to be minimal.
A smooth map $f:M\to N$ is called \emph{minimal},
if its graph $\Gamma(f)$ is a minimal submanifold of the product
space $(M\times N,\gMN)$. \\

On the submanifold, the \emph{Gau{\ss} equation}
\begin{align}
	\nonumber \bigl( \Rind - F^*\Rm_{M\times N} \bigr)(u_1,v_1,u_2,v_2) &=  \gMN\bigl( \A(u_1,u_2), \A(v_1,v_2) \bigr) \\
		\label{eq:Gauss} & \quad - \gMN\bigl( \A(u_1,v_2), \A(v_1,u_2) \bigr)
\end{align}
and the \emph{Codazzi equation}
\[
	(\nabla_{u}\A)(v,w) - (\nabla_{v}\A)(u,w) = \Rm_{M\times N}\bigl( \dd F(u), \dd F(v) \bigr) \dd F(w) - \dd F\bigl(\Rind(u,v)w\bigr)
\]
hold, where the induced connection on the bundle $F^*T(M\times N)\otimes T^*M \otimes T^*M$
is defined as
\[
	(\nabla_{u}\A)(v,w) \coloneqq \D_{\dd F(u)}(\A(v,w)) - \A(\nabla_uv,w) - \A(v,\nabla_uw) \,.
\]

\subsection{Singular Value Decomposition}
\label{sec:SVD}

We recall the singular value decomposition theorem  and closely follow \cite{SHS13}*{Section 3.2}. \\

Fix a point $x\in M$, and let
\[
	\sv_1^2(x) \le \sv_2^2 \le \dotsb \le \sv_m^2(x)
\]
be the eigenvalues of $f^{*}\gN$ with respect to $\gM$. The
corresponding values $\sv_i\ge 0$, $i\in\{1,\dotsc,m\}$, are called the
\emph{singular values} of the differential $\dd f$ of $f$ and give rise
to continuous functions on $M$. Let
\[
	r = r(x) \coloneqq \rank \dd f(x) \,.
\]
Obviously, $r\le \min\{m,n\}$ and $\sv_1(x)=\dotsb=\sv_{m-r}(x)=0$. At the
point $x$ consider an orthonormal basis $\{\alpha_{1},\dotsc,\alpha_{m-r};\alpha_{m-r+1},\dotsc,\alpha_{m}\}$
with respect to $\gM$ which diagonalizes $f^{*}\gN$. Moreover, at $f(x)$ 
consider a basis $\{\beta_{1},\dotsc,\beta_{n-r};\beta_{n-r+1},\dotsc,\beta_{n}\}$ that is
orthonormal
with respect to $\gN$, such that
\[
	\dd f(\alpha_{i}) = \sv_i(x) \beta_{n-m+i}
\]
for any $i\in\{m-r+1,\dotsc,m\}$. This procedure is called the
\emph{singular value decomposition} of the differential $\dd f$. \\

Now let us construct a special basis for the tangent and the normal
space of the graph in terms of the singular values. The vectors
\[
	\widetilde{e}_{i} \coloneqq 
	\begin{cases} 
		\alpha_{i} \oplus 0\,, & 1\le i\le m-r\,, \\ 
		\frac{1}{\sqrt{1+\sv_{i}^{2}(x)}}\bigl( \alpha_{i} \oplus \sv_{i}(x)\beta_{n-m+i} \bigr)\,, &m-r+1 \le i \le m\,, 
	\end{cases}
\]
form an orthonormal basis with respect to the metric $\gMN$ of 
the tangent space $\dd F(T_{x}M)$ of the graph $\Gamma(f)$ at $x$. It follows
that with respect to the induced metric $\gind$, the vectors
\[
	e_i \coloneqq \frac{1}{\sqrt{1+\sv_i^2(x)}} \alpha_i
\]
form an orthonormal basis of $T_xM$.
Moreover,
the vectors
\[
	\xi_{i} \coloneqq
	\begin{cases}
		0 \oplus \beta_{i}\,, & 1 \le i \le n-r\,, \\
		\frac{1}{\sqrt{1+\sv_{i+m-n}^{2}(x)}}\bigl(-\sv_{i+m-n}(x)\alpha_{i+m-n} \oplus \beta_{i}\bigr)\,, & n-r+1\le i\le n\,,
	\end{cases}
\]
form an orthonormal basis with respect to $\gMN$ of the
normal space $T^{\perp}_{x}M$ of the graph $\Gamma(f)$ at the point
$x$. From the formulae above, we deduce that
\[
	\sMN\bigl(\widetilde{e}_{i},\widetilde{e}_{j}\bigr) = \sind(e_i,e_j) = \frac{1-\sv_{i}^{2}(x)}{1+\sv_{i}^{2}(x)}\delta_{ij} \,, \qquad 1\le i,j\le m \,.
\]
Therefore, the eigenvalues of the $2$-tensor $\sind$ with respect to $\gind$
are given by
\begin{equation}
	\label{eq:sSV}
	\frac{1-\sv_{1}^{2}(x)}{1+\sv_{1}^{2}(x)} \ge \cdots \ge \frac{1-\sv_{m-1}^{2}(x)}{1+\sv_{m-1}^{2}(x)} \ge \frac{1-\sv_{m}^{2}(x)}{1+\sv_{m}^{2}(x)} \,.
\end{equation}
Moreover,
\begin{equation}
	\label{eq:sOnNormalBdl}
	\sMN(\xi_{i},\xi_{j}) =
		\begin{cases}
			- \delta_{ij}\,, & 1 \le i \le n-r \,, \\
			- \frac{1-\sv_{i+m-n}^{2}(x)}{1+\sv_{i+m-n}^{2}(x)}\delta_{ij} \,, & n-r+1\le i \le n\,.
		\end{cases}
\end{equation}
Thus, if there exists a positive constant $\varepsilon$ such that $\sind\ge\varepsilon\gind$, then
$\sN\le - \varepsilon\gNorm$. Furthermore,
\[
	\sMN(\widetilde{e}_{m-r+i},\xi_{n-r+j}) = - \frac{2\sv_{m-r+i}(x)}{1+\sv_{m-r+i}^{2}(x)} \delta_{ij} \,,\qquad 1\le i,j\le r \,.
\]

\section{Mean Curvature Flow}
\label{sec:MCF}

Let us consider the case where $M=\Rb^m$ and 
$N$ is a complete, non-compact Riemannian manifold 
with bounded geometry satisfying $\sec_N\le -\sigma < 0$ for some $\sigma>0$.
Further, let $f:\Rb^m\to N$ be a smooth map and $T>0$. We say that a family
of maps $F:\Rb^m\times[0,T)\to\Rb^m\times N$ evolves under the mean curvature flow,
if for all $x\in\Rb^m$
\begin{equation}
	\label{eq:MCF}
	\begin{cases} \partial_t F(x,t) = \Hv(x,t) \,, \\ F(x,0) = \bigl(x,f(x)\bigr) \,. \end{cases}
\end{equation}

\subsection{Short-time Existence}

Using that $N$ has bounded geometry, 
there is a neighborhood $U\subset \Rb^m\times N$ of
$\Gamma(f_0)$ and $\Omega\subset\Rb^n$, such that $U$
is simply-connected and such that there is a 
diffeomorphism $\psi:U\to\Rb^m\times\Omega$.
Let us denote the local coordinates induced by $\psi$ on $N$
(depending on the point $x\in\Rb^m$)
by $\{y^1,\dotsc,y^n\}$.
With this identification, if $F:\Rb^m\to \Rb^m\times N$ is a graph over $\Rb^m$
and $F(\Rb^m)\subset \Rb^m\times U$,
we may equivalently consider $\psi \circ F : \Rb^m\to\Rb^m\times\Omega$,
which then has the form $\psi\circ F(x) = (x,f(x))$ for some map $f:\Rb^m\to \Omega$.

Using standard coordinates $\{x^1,\dotsc,x^m\}$ 
on $\Rb^m$ and denoting the local coordinates on $\Rb^m\times U$
collectively by $\{z^1,\dotsc,z^{m+n}\}=\{x^1,\dotsc,x^m,y^1,\dotsc,y^n\}$, 
the evolution equation for the mean curvature flow in this chart is given by
\begin{equation}
	\label{eq:localMCF}
	\partial_t F_t = \sum_{i,j=1}^m \gind^{ij} \left( \partial_{ij}^2 F_t - \sum_{l=1}^{m} \Gamma^l_{ij} \partial_l F_t + \sum_{a,b,c=1}^{m+n} \bigl( \Gamma^{\RgMN} \bigr)_{ab}^c (\partial_i F_t^a) (\partial_j F_t^b) \frac{\partial}{\partial z^c} \right) \,.
\end{equation}
Assuming a graphical solution exists up to time $T>0$,
we may choose a time-dependent diffeomorphism $\varphi : \Rb^m\times[0,T) \to\Rb^m$, such that
$F_t\circ \varphi_t(x)=(x,f_t(x))$ for each $t\in[0,T)$.
Also using $\nabla^{\RgMN} = \D \oplus \nabla^{\RgN}$ (where $\D$ denotes the
flat connection on $\Rb^m$),
the system \eqref{eq:localMCF} with initial condition $F_0(x)=(x,f(x))$ reduces to
\begin{equation}
	\label{eq:pMCF}
	\begin{cases}
		\partial_t f_t = \sum_{i,j=1}^m \tgind^{ij} \left( \partial^2_{ij} f_t + \sum_{a,b=1}^n (\Gamma^{\gN})_{ab} (\partial_i f_t^a) (\partial_j f_t^b) \right) \,, \\
		f_0(x) = f(x) \,,
	\end{cases}
\end{equation}
where here $\tgind^{ij}$ are the components of the inverse of $\tgind \coloneqq \RgM + f_t^*\gN$
and $\Gamma^{\gN}$ are the Christoffel symbols of the metric $\gN$.
Since $\Gamma^{\gN}$ only depends on the geometry of $(N,\gN)$, the second term 
only contributes to lower orders.

If
\eqref{eq:pMCF} has a smooth solution $f:\Rb^m\times[0,T)\to N$, then the mean curvature flow \eqref{eq:MCF}
has a smooth solution $F:\Rb^m\times[0,T)\to\Rb^m\times N$ given by the family of graphs
\[
	\Gamma\bigl(f(\cdot,t)\bigr) = \bigl\{ \bigl(x,f(x,t)\bigr) : x\in\Rb^m \bigr\} \,,
\]
up to tangential diffeomorphisms (see e.\,g.\ \cite{Bra78}*{Chapter 3.1}).

For \eqref{eq:pMCF}, we have the following short-time existence result.

\begin{theorem}
	\label{thm:ShortTimeEx}
	Let $(N,\gN)$ be a complete Riemannian manifold with bounded geometry.
	Further, let $f_0:\Rb^m\to N$ be a smooth function, such that for each $k\ge 0$ we have 
	$\sup_{x\in\Rb^m}\|\nabla^k \dd f_0(x)\|\le C_k$ for some finite constants $C_k$.
	Then \eqref{eq:pMCF} has a short-time smooth solution $f$ on $\Rb^m\times[0,T)$
	for some $T>0$ with initial condition $f_0$,
	such that $\sup_{x\in\Rb^m}\|\nabla^k \dd f_t(x)\| < \infty$ for every $k\ge 0$ and $t\in[0,T)$.
\end{theorem}

\begin{proof}
	By the above construction, we obtain
	$U\subset \Rb^m\times N$, $\Omega\subset\Rb^n$ and a diffeomorphism 
	$\psi:U \to \Rb^m\times \Omega$. We may choose $\psi$, such that
	the coordinates induced on $N$ are normal coordinates.	
	By the bounded geometry assumption on $N$, we may further assume that $U$ is chosen
	such that the Christoffel symbols $\Gamma^{\gN}$ are uniformly bounded in $\{p\}\times\Omega$
	for all $p\in\Rb^m$ \cites{Kaul76,Eichhorn91}.
	Thus, it is sufficient to obtain a short-time solution to
	\eqref{eq:pMCF} in $\Rb^m\times \Omega$.
	Since equation \eqref{eq:pMCF} is strongly parabolic and only differs
	by
	bounded lower-order terms from the mean curvature flow system in flat space,
	the claim follows in the same way as \cite[Proposition 5.1]{CCH12}.
	In particular, for short times the solution stays inside $\Omega$, so that
	it maps to a solution to the mean curvature flow in $U$.
\end{proof}

In the sequel, we will consider a special kind of solution to \eqref{eq:MCF}.
\begin{definition}
	Let $F_t(x)$ be a smooth solution to the system \eqref{eq:MCF} on $\Rb^m\times[0,T)$ for some
	$0<T\le\infty$, such that for each $t\in[0,T)$ and non-negative integer $k$, 
	the submanifold $F_t(\Rb^m)\subset\Rb^m\times N$ satisfies
	\begin{gather}
		\label{eq:BddGeom1}
		\sup_{x\in\Rb^m} \|\nabla^k \A(x,t) \| < \infty \,, \\
		\label{eq:BddGeom2}
		C_1(t) \RgM \le \gind \le C_2(t) \RgM \,,
	\end{gather}
	where $C_1(t)$ and $C_2(t)$ for each $t\in[0,T)$ are finite, positive constants
	depending only on $t$. Then we will say that the family of embeddings 
	$\{F_t\}_{t\in I}$
	has \emph{bounded geometry}.
\end{definition}

\begin{definition}
	Let $f_t(x)$ be a smooth solution to the system \eqref{eq:pMCF} on $\Rb^m\times[0,T)$ for
	some $0<T\le\infty$, such that for each $t\in[0,T)$ and positive integer $k$ the estimate
	\[
		\sup_{x\in\Rb^m} \| \nabla^{k-1} \dd f_t(x) \| < \infty
	\]
	holds. Then we will say that $f_t(x)$ has \emph{bounded geometry} for every $t\in[0,T)$.
\end{definition}

\subsection{Graphs}
We recall some important notions in the graphic case, where we follow the presentation in \cite[Section 3.1]{SHS14}.

Let $f_0:\Rb^m\to N$ denote a smooth map, such that $F_0(x)\coloneqq (x,f_0(x))$
has bounded geometry. Then theorem \ref{thm:ShortTimeEx} ensures that the system
\eqref{eq:pMCF} has a short-time solution with initial data $f_0(x)$
on a time interval $[0,T)$ for some positive maximal time $T>0$. Further, there
exists a diffeomorphism $\varphi_t:\Rb^m\to\Rb^m$, such that 
\begin{equation}
	\label{eq:DiffeoRel}
	F_t\circ \varphi_t(x)=(x,f_t(x)) \,,
\end{equation}
where $F_t(x)$ is a solution of \eqref{eq:MCF}.

To obtain the converse of this statement,
let $\Omega_{\Rb^m}$ be the volume form on $\Rb^m$ and extend it to a parallel $m$-form on $\Rb^m\times N$
by pulling it back via the natural projection $\RpiM:\Rb^m\times N\to\Rb^m$, that is, consider
the $m$-form $\RpiM^*\Omega_{\Rb^m}$. Define the time-dependent smooth function $u:\Rb^m\times[0,T)\to\Rb$
by
\[
	u \coloneqq \star \Omega_t \,,
\]
where $\star$ is the Hodge star operator with respect to the induced metric $\gind$ and
\[
	\Omega_t \coloneqq F_t^*\bigl( \RpiM^*\Omega_{\Rb^m} \bigr) = (\RpiM \circ F_t)^* \Omega_{\Rb^m}\,.
\]
The function $u$ is the Jacobian of the projection map from $F_t(\Rb^m)$ to $\Rb^m$. From
the implicit mapping theorem it follows that $u>0$ if and only if there exists a diffeomorphism
$\varphi_t:\Rb^m\to\Rb^m$ and a map $f_t:\Rb^m\to N$, such that \eqref{eq:DiffeoRel} holds,
i.\,e.\ $u$ is positive precisely if the solution of the mean curvature flow remains a graph.
By theorem \ref{thm:ShortTimeEx},
the solution will stay a graph at least in a short time
interval $[0,T)$.

\subsection{Evolution Equations}

Let us consider the evolution of the tensors defined in section \ref{sec:GeomOfGraphs} under the mean curvature flow.
The evolution of the tensor $\sind$ is essentially calculated in \cite[Lemma 3.1]{SHS14}
and is given by the following statement.

\begin{lemma}
	\label{lem:sEvolEq}
	The evolution of the tensor $\sind$ for $t\in[0,T)$
	is given by the formula
	\begin{align*}
		\left( \nabla_{\dt} \sind - \Lap \sind\right)(v,w) = & - \sind (\Ric v, w ) - \sind( w, \Ric v ) \\
			& - 2 \sum_{k=1}^m \RsMN\bigl( \A(e_k,v), \A(e_k,w) \bigr) \\
			& - 2 \sum_{k=1}^m f_t^*\RmN(e_k,v,e_k,w) \,,
	\end{align*}
	where $\{e_1,\dotsc,e_m\}$ is any orthonormal frame with respect to $\gind$
	and the \emph{Ricci operator} is given by
	\[
		\Ric v \coloneqq - \sum_{k=1}^m \Rind(e_k,v) e_k \,.
	\]
\end{lemma}

For the tensor $\sN$, from \cite[Lemma 3.3]{SHS15} we have the following evolution equation.

\begin{lemma}
	\label{lem:sNEvolEq}
	Let $\xi$ be a unit vector normal to the evolving submanifold at a fixed point
	$(x_0,t_0)$ in space-time. Then
	\begin{align*}
		&\bigl( \nN_{\dt}\sN - \Lap^{\perp}\sN \bigr)(\xi,\xi) \\
			& \quad = 2 \sum_{i,j=1}^m \A_{\xi}(e_i,e_j) \RsMN\bigl( \A(e_i,e_j), \xi \bigr) - 2 \sum_{i,j,k=1}^m \A_{\xi}(e_i,e_j) \A_{\xi}(e_i,e_k) \sind(e_j,e_k) \\
			& \qquad - 2 \sum_{i,j=1}^m \RmMN( \dd F(e_i), \dd F(e_j), \dd F(e_i),\xi) \RsMN(\dd F(e_j),\xi)
	\end{align*}
	for any $\gind$-orthonormal basis $\{e_1,\dotsc,e_m\}$ of $T_{x_0}\Rb^m$.
\end{lemma}

Let us define the symmetric $2$-tensor $\upvartheta \in \Sym\bigl( F^*T^*(\Rb^m\times N) \otimes F^*T^*(\Rb^m\times N) \bigr)$
by setting
\[
	\upvartheta(\xi,\eta) \coloneqq \Hv_{\prN(\xi)} \Hv_{\prN(\eta)} \,.
\]
By \cite[Lemma 3.4]{SHS15}, this tensor satisfies the following evolution equation.

\begin{lemma}
	\label{lem:thetaEvolEq}
	The symmetric $2$-tensor $\upvartheta$ evolves under the mean curvature flow according to the formula
	\begin{align*}
		\bigl( \nN_{\dt} \upvartheta - \Lap^{\perp}\upvartheta \bigr)(\xi,\xi) &= 2 \sum_{i,j=1}^m \A_{\Hv}(e_i,e_j) \A_{\xi}(e_i,e_j) \Hv_{\xi} - 2 \sum_{i=1}^m \bigl\< \nN_{e_i}\Hv,\xi \bigr\>^2 \\
			& \quad - 2 \sum_{i=1}^m \RmMN\bigl( \Hv, \dd F(e_i), \dd F(e_i), \xi \bigr) \Hv_{\xi}
	\end{align*}
	for any vector $\xi$ in the normal bundle of the submanifold.
\end{lemma}

\section{A Priori Estimates}

\subsection{Preserved Quantities}

Following the idea in \cite{CCH12}, we will need the function
\begin{equation}
	\phi_R(x) \coloneqq 1 + \frac{\|x\|^2_{\Rb^m}}{R^2} \,,
\end{equation}
where $\|\cdot\|_{\Rb^m}$ denotes the Euclidean norm on $\Rb^m$ and $R>0$ is a constant which will
be chosen later.

\begin{lemma}[{\cite[Lemma 4.1]{Lub16}}]
	Let $F(x,t)$ be a smooth solution to \eqref{eq:MCF} with bounded geometry and
	assume there exists $\varepsilon>0$, such that $\sind-\varepsilon \gind\ge 0$
	for any $t\in[0,T)$. Fix any $T'\in [0,T)$ and $(x_0,t_0)\in\Rb^m\times[0,T']$.
	Then for any tangent vector $v$ and any normal vector $\xi$ at $(x_0,t_0)$,
	the following estimates hold,
	\begin{align}
		-c(T') \frac{\|x_0\|_{\Rb^m}}{R^2} \sind(v,v) &\le \< \nabla \phi_R, (\nabla\sind)(v,v) \> \le c(T') \frac{\|x_0\|_{\Rb^m}}{R^2} \sind(v,v) \,, \\
		c(T') \frac{\|x_0\|_{\Rb^m}}{R^2} \sN(\xi,\xi) &\le \< \nabla \phi_R, (\nN\sN)(\xi,\xi)\> \le -c(T') \frac{\|x_0\|_{\Rb^m}}{R^2} \sN(\xi,\xi) \,, \\
		| \Lap \phi_R | &\le c(T') \left( \frac{1}{R^2} + \frac{\|x_0\|_{\Rb^m}}{R^2} \right) \,,
	\end{align}
	where $c(T')\ge 0$ is a constant depending only on $T'$.
\end{lemma}

To obtain the preservation of the length-decreasing property, for any
$R,\eta>0$ let us further set
\[
	\uppsi_{|(x,t)} \coloneqq \ee^{\eta t} \phi_R(x) \sind_{|(x,t)} - \varepsilon \gind_{|(x,t)} \,.
\]

\begin{lemma}
	Under the mean curvature flow, the tensor $\uppsi$ evolves according to the equation
	\begin{align*}
		\bigl( \nabla_{\dt}\uppsi &{}- \Lap \uppsi \bigr)(u_1,u_2) \\
			&{} = - \uppsi( \Ric u_1, u_2) - \uppsi( u_1, \Ric u_2 ) \\
			&{} \phantom{{}={}} + 2 \varepsilon \sum_{k=1}^m \< \A(u_1,e_k), \A(u_2,e_k) \> \\
			&{} \phantom{{}={}} - 2 \ee^{\eta t} \phi_R \sum_{k=1}^m \RsMN\bigl( \A(u_1,e_k), \A(u_2,e_k) \bigr) \\
			&{} \phantom{{}={}} - 2 \ee^{\eta t}\phi_R \sum_{k=1}^m f^*_t\RmN(e_k,u_1,e_k,u_2) \\
			&{} \phantom{{}={}} - \ee^{\eta t}\Bigl\{ (\Lap \phi_R) \sind(u_1,u_2) + 2 \<\nabla\phi_R, (\nabla\sind)(u_1,u_2)\> - \eta \phi_R \sind(u_1,u_2) \Bigr\}
	\end{align*}
	for any $u_1,u_2\in\Gamma(T\Rb^m)$ and any local frame $\{e_1,\dotsc,e_m\}$ which is
	orthonormal with respect to $\gind$.
\end{lemma}

\begin{proof}
	The proof is the same as the proof of \cite[Lemma 4.2]{Lub16}, but one 
	needs to take the additional curvature term occuring in lemma \ref{lem:sEvolEq} into consideration.
\end{proof}

\begin{lemma}
	\label{lem:LenPres}
	Let $F(x,t)$ be a smooth solution to \eqref{eq:MCF} with bounded geometry. 
	Assume there exists $\varepsilon>0$ with $\sind-\varepsilon\gind\ge 0$ at $t=0$.
	Then it is $\sind-\varepsilon\gind\ge 0$ for all $t\in[0,T)$.
\end{lemma}

\begin{proof}
	We first assume that also $\sind-\frac{\varepsilon}{2}\ge 0$ on
	$[0,T') \subset [0,T)$. 
	By assumption, the curvature of $N$ 
	satisfies $\sec_N\le 0$, so that the curvature term in the evolution equation
	of $\uppsi$ with respect to any $v\in\Gamma(T\Rb^m)$ satisfies
	\[
		- 2 \ee^{\eta t}\phi_R \sum_{k=1}^m f_t^* \RmN(e_k,v,e_k,v) \ge 0 \,.
	\]
	Thus, we may argue in exactly the same way as in the proof of \cite[Lemma 4.3]{Lub16},
	which shows $\sind-\varepsilon\gind\ge 0$ is preserved on $[0,T')$.
	
	Finally, the additonal assumption on $\sind$ is removed as in \cite[Lemma 4.4]{Lub16}, and
	the claim follows.
\end{proof}

It immediately follows that a smooth length-decreasing map $f:\Rb^m\to N$ evolves
through length-decreasing maps $f_t:\Rb^m\to N$ under the mean curvature flow
\cite[Lemma 4.5]{Lub16}. \\

The next step will be to show that the mean curvature vector of the graph remains bounded under
the mean curvature flow. For this, let us set
\[
	\upchi \coloneqq -\ee^{\eta t} \phi_R \sN - \varepsilon_2 \uptheta \,.
\]

\begin{lemma}
	\label{lem:ChiEvolEq}
	For any unit vector $\xi$ normal to the evolving submanifold at a fixed
	point $(x_0,t_0)\in \Rb^m\times [0,T)$, the tensor $\upchi$ satisfies the equation
	\begin{align*}
		& \bigl( \nN_{\dt} \upchi - \LapN\upchi \bigr)(\xi,\xi) \\
			& \quad = \ee^{\eta t_0} \Bigl\{ - \eta \phi_R \sN(\xi,\xi) + (\Lap\phi_R)\sN(\xi,\xi) + 2 \< \nabla \phi_R, (\nN\sN)(\xi,\xi) \> \Bigr\} \\
				& \qquad - 2 \ee^{\eta t_0} \phi_R \left\{ \sum_{i,j=1}^m \A_{\xi}(e_i,e_j)\RsMN\bigl( \A(e_i,e_j), \xi \bigr) \right. \\
					& \qquad \qquad \qquad \qquad \left. - \sum_{i,j,k=1}^m \A_{\xi}(e_i,e_j) \A_{\xi}(e_i,e_k) \sind(e_j,e_k) \right\} \\
				& \qquad - 2 \varepsilon_2 \left\{ \sum_{i,j=1}^m \A_{\Hv}(e_i,e_j) \A_{\xi}(e_i,e_j) \Hv_{\xi} - \sum_{i=1}^m \bigl\< \nN_{e_i}\Hv,\xi \bigr\>^2 \right\} \\
				& \qquad + 2 \ee^{\eta t_0} \phi_R \sum_{i,j=1}^m \RmMN\bigl( \dd F(e_i), \dd F(e_j), \dd F(e_i), \xi\bigr) \RsMN\bigl( \dd F(e_j), \xi \bigr) \\
				& \qquad + 2 \varepsilon_2 \sum_{i=1}^m \RmMN\bigl( \Hv, \dd F(e_i), \dd F(e_i), \xi) \Hv_{\xi}
	\end{align*}
	under the mean curvature flow.
\end{lemma}

\begin{proof}
	We calculate
	\begin{align*}
		\bigl(\nN_{\dt}\upchi - \LapN\upchi\bigr)(\xi,\xi) &= -\ee^{\eta t}\phi_R \bigl(\nN_{\dt}\sN - \LapN\sN\bigr)(\xi,\xi) - \varepsilon_2 \bigl(\nN_{\dt}\uptheta-\LapN\uptheta\bigr)(\xi,\xi) \\
			& \qquad - \eta \ee^{\eta t} \phi_R \sN(\xi,\xi) + \ee^{\eta t}(\Lap\phi_R)\sN(\xi,\xi) \\
			& \qquad + 2 \ee^{\eta t} \< \nabla \phi_R, (\nN\sN)(\xi,\xi)\> \,.
	\end{align*}
	The claim follows from the evolution equations for $\sN$ and $\uptheta$ in lemmas \ref{lem:sNEvolEq} and \ref{lem:thetaEvolEq}.
\end{proof}

\begin{lemma}
	\label{lem:MCVEst}
	Let $F(x,t)$ be a smooth solution to \eqref{eq:MCF} with bounded geometry and suppose
	$\sind-\varepsilon_1\gind\ge 0$ on $[0,T)$ for some $\varepsilon_1>0$. Then there exists
	a constant $\varepsilon_2>0$ depending on $\varepsilon_1$, the dimension $m=\dim \Rb^m$
	and the geometry of $N$, such that
	\[
		\sN + \varepsilon_2 \uptheta \le 0
	\]
	on $\Rb^m\times[0,T)$.
\end{lemma}

\begin{proof}
	We follow the same strategy as in the proof of
	\cite[Lemma 5.2]{Lub16}. Fix any $T'\in[0,T)$. We will
	first show that we can choose $R_0>0$, such that $\upchi\ge 0$ on $\Rb^m\times[0,T')$
	for all $R\ge R_0$.
	
	Suppose $\upchi$ is not positive on $\Rb^m\times[0,T']$ for some $R\ge R_0$. Then,
	as $\upchi>0$ on $\Rb^m\times\{0\}$, $\sind-\varepsilon_1\gind\ge 0$ (and thus
	$\sN+\varepsilon_1\gNorm\le 0$) on $[0,T)$, $\phi_R(x)\to\infty$ as $\|x\|\to\infty$
	and by the bounded geometry condition \eqref{eq:BddGeom1}, it follows that $\upchi>0$
	outside some compact set $K\subset \Rb^m$ and all $t\in[0,T']$. We conclude
	that there exists $(x_0,t_0)\in K\times[0,T']$, such that $\upchi$ has a zero
	eigenvalue at $(x_0,t_0)$ and that $t_0$ is the first such time. In other
	words, we have $\upchi_{|(x_0,t_0)}(\xi,\eta)=0$ for some nonzero vector $\xi$
	and all vectors $\eta$, and $\upchi>0$ on $\Rb^m\times[0,t_0)$. Extend $\xi$
	to a local smooth vector field.
	By the second derivative criterion, at $(x_0,t_0)$ we have
	\begin{equation}
		\label{eq:ChiEst}
		\upchi(\xi,\eta)=0 \,, \quad (\nN\upchi)(\xi,\xi)=0\,, \quad (\nN_{\dt}\upchi)(\xi,\xi)\le 0 \quad\text{and} \quad (\LapN\upchi)(\xi,\xi)\ge 0
	\end{equation}
	for all $\eta\in T_{x_0}\Rb^m$. Let us set
	\begin{align*}
		\Ac &\coloneqq \ee^{\eta t} \Big\{ - \eta \phi_R \sN(\xi,\xi) + (\Lap \phi_R)\sN(\xi,\xi) + 2 \<\nabla\phi_R, (\nN\sN)(\xi,\xi)\> \Bigr\} \,, \\
		\Bc &\coloneqq -2\ee^{\eta t}\phi_R \sum_{i,j=1}^m \A_{\xi}(e_i,e_j)\RsMN\bigl( \A(e_i,e_j), \xi \bigr) \\
				&\qquad + 2\ee^{\eta t}\phi_R \sum_{i,j,k=1}^m \A_{\xi}(e_i,e_j) \A_{\xi}(e_i,e_k) \sind(e_j,e_k) \\
			& \qquad - 2 \varepsilon_2 \left\{ \sum_{i,j=1}^m \A_{\Hv}(e_i,e_j) \A_{\xi}(e_i,e_j) \Hv_{\xi} - \sum_{i=1}^m \bigl\< \nN_{e_i} \Hv, \xi \bigr\>^2 \right\} \,, \\
		\Cc &\coloneqq 2 \ee^{\eta t} \phi_R \sum_{i,j=1}^m \RmMN\bigl( \dd F(e_i), \dd F(e_j), \dd F(e_i), \xi \bigr) \RsMN\bigl( \dd F(e_j), \xi \bigr) \\
			& \qquad + 2 \varepsilon_2 \sum_{i=1}^m \RmMN\bigl( \Hv, \dd F(e_i), \dd F(e_i), \xi \bigr) \Hv_{\xi} \,.
	\end{align*}
	Then at $(x_0,t_0)$ it is
	\[
		0 \stackrel{\text{Eq.\ \eqref{eq:ChiEst}}}{\ge} (\nN_{\dt}\upchi)(\xi,\xi) - (\LapN\upchi)(\xi,\xi) = \Ac + \Bc + \Cc \,.
	\]
	The proof of \cite[Lemma 5.2]{Lub16} yields that
	that we can choose $R_0>0$ (depending
	on $\eta$ and $T'$) large enough,
	so that
	\[
		\Ac > 0 \qquad \text{for any $x_0$ and for all $R\ge R_0$} \,,
	\]
	and, furthermore, if $\varepsilon_2$ satisfies 
	$0<\varepsilon_2\le \frac{2\varepsilon_1}{m}$, we obtain 
	\[
		\Bc \ge 2 \ee^{\eta t_0}\phi_R \varepsilon_1 \sum_{i,j=1}^m \A_{\xi}^2(e_i,e_j) \ge 2\ee^{\eta t_0} \phi_R \frac{\varepsilon_1}{m} \Hv_{\xi}^2 \stackrel{\phi_R\ge 1}{\ge} 2\ee^{\eta t_0} \frac{\varepsilon_1}{m} \Hv_{\xi}^2 \,.
	\]
	It remains to show that $\Bc+\Cc\ge 0$.
	Eq.\ \eqref{eq:ChiEst} implies $\ee^{\eta t_0} \phi_R \sN(\xi,\eta)=-\varepsilon_2 \uptheta(\xi,\eta)$ at $(x_0,t_0)$, and we calculate
	\begin{align*}
		\Cc &\stackrel{\phantom{\text{Eq.\ \eqref{eq:ChiEst}}}}{=} 2 \ee^{\eta t_0} \phi_R \sum_{i,j=1}^m \RmMN\bigl( \dd F(e_i), \dd F(e_j), \dd F(e_i), \xi \bigr) \RsMN\bigl( \dd F(e_j), \xi \bigr) \\
				& \qquad \qquad + 2 \varepsilon_2 \sum_{i=1}^m \sum_{k=1}^n \RmMN\bigl( \xi_k, \dd F(e_i), \dd F(e_i), \xi \bigr) \Hv_{\xi_k} \Hv_{\xi} \\
			&\stackrel{\text{Eq.\ \eqref{eq:ChiEst}}}{=} 2 \ee^{\eta t_0} \phi_R \sum_{i,j=1}^m \RmMN\bigl( \dd F(e_i), \dd F(e_j), \dd F(e_i), \xi \bigr) \RsMN\bigl( \dd F(e_j), \xi \bigr) \\
				& \qquad \qquad - 2 \ee^{\eta t_0}\phi_R \sum_{i=1}^m \sum_{k=1}^n \RmMN\bigl( \xi_k, \dd F(e_i), \dd F(e_i), \xi \bigr) \sN( \xi_k, \xi ) \,.
	\end{align*}
	From the bounded geometry of $N$ and the boundedness of the singular values of the map $f_t$ we infer that there is
	a finite constant $C>0$ depending only on the geometry of $N$ and the singular values, such that $\Cc$ may
	be estimated as
	\[
		\Cc \ge - 2 C \ee^{\eta t_0} \phi_R \,.
	\]
	Again using Eq.\ \eqref{eq:ChiEst} yields
	\[
		\varepsilon_2 \Hv_{\xi}^2 = - \ee^{\eta t_0} \phi_R \sN(\xi,\xi) \ge \ee^{\eta t_0} \phi_R \varepsilon_1 > 0 \,,
	\]
	so that
	\[
		\frac{\varepsilon_2}{\varepsilon_1} \Hv_{\xi}^2 \ge \ee^{\eta t_0} \phi_R > 0
	\]
	and accordingly
	\[
		\Cc \ge - 2 C \frac{\varepsilon_2}{\varepsilon_1} \Hv_{\xi}^2 \,.
	\]
	From the estimate for $\Bc$, we then get
	\[
		\Bc + \Cc \ge 2 \ee^{\eta t_0} \left( \frac{\varepsilon_1}{m} - C \frac{\varepsilon_2}{\varepsilon_1} \right) \Hv_{\xi}^2 \,.
	\]
	Now choose $\varepsilon_2$, such that $0<\varepsilon_2<\min\bigl\{ \frac{2\varepsilon_1}{m}, \frac{\varepsilon_1^2}{mC}\bigr\}$.
	Then it is $\Bc+\Cc > 0$ and furthermore $\Ac+\Bc+\Cc>0$, 
	which contradicts Eq.\ \eqref{eq:ChiEst}, so that $\upchi<0$ along the flow.
	
	The claim of the lemma follows by first sending $R\to\infty$, then $\eta\to 0$, 
	and finally $T'\to T$.
\end{proof}

\begin{corollary}
	Under the mean curvature flow, the mean curvature vector of the graph of
	a length-decreasing map $f_t:\Rb^m\to N$ stays bounded, i.\,e.\ there
	is a constant $C>0$, such that
	\[
		\|\Hv\|^2\le C \,.
	\]
\end{corollary}

\begin{proof}
	Let $\{\xi_1,\dotsc,\xi_n\}$ be an orthonormal basis of $T_{x_0}^{\perp}\Rb^m$. Using
	Lemma \ref{lem:MCVEst}, we obtain
	\[
		\varepsilon_2 \|\Hv\|^2 = \varepsilon_2 \sum_{k=1}^n \Hv_{\xi_k}^2 \le - \sum_{k=1}^n \sN(\xi_k,\xi_k) \le n \,,
	\]
	which establishes the claim.
\end{proof}

\subsection{Improved First-Order Estimate}

In some cases, namely if $m>1$ and $\inf_{\Rb^m\times\{0\}}\tr(\sind)<m-1$, 
the estimate on the singular values of $\dd f_t$ can be improved.

\begin{lemma}
	\label{lem:trEst}
	Assume $N$ satisfies the curvature bound
	\[
		\sec_N \le -\sigma < 0
	\]
	for some constant $\sigma>0$. If $f_t:\Rb^m\to N$ is a family of strictly
	length-decreasing maps that evolve under the mean curvature flow, the
	trace of the tensor $\sind$ satisfies
	\[
		(\partial_t - \Delta) \tr(\sind) \ge 2 \sigma \sum_{k=1}^m \frac{2}{1+\lambda_k^2} \frac{\lambda_k^2}{1+\lambda_k^2} \left\{ \sum_{l=1}^m \frac{\lambda_l^2}{1+\lambda_l^2} - \frac{\lambda_k^2}{1+\lambda_k^2} \right\} \,.
	\]
\end{lemma}

\begin{proof}
	Using lemma \ref{lem:sEvolEq} and the Gau{\ss} equation \eqref{eq:Gauss}, we obtain the evolution equation
	for the trace of the tensor $\sind$,
	\begin{multline*}
		(\partial_t - \Delta) \tr(\sind) = -2 \sum_{k,l=1}^m \left( \RsMN - \frac{1-\lambda_k^2}{1+\lambda_k^2} \RgMN \right) \bigl( \A(e_k,e_l), \A(e_k,e_l) \bigr) \\
			- 2 \sum_{k,l=1}^m \frac{2}{1+\lambda_k^2} f^*\RmN(e_k,e_l,e_k,e_l) \,.
	\end{multline*}
	Since the maps $f_t$ are length-decreasing, it is $\RsMN\bigl( \A(u,v), \A(u,v) \bigr)\le 0$
	for any $u,v\in\Gamma(T\Rb^m)$ and also $\lambda_k^2\le 1$. Consequently,
	\begin{align*}
		(\partial_t - \Delta) \tr(\sind) &\ge - 2 \sum_{k,l=1}^m \frac{2}{1+\lambda_k^2} f^*\RmN(e_k,e_l,e_k,e_l) \\
			& = -2 \sum_{k\ne l} \frac{2}{1+\lambda_k^2} \sec_N\bigl( \dd f(e_k) \wedge \dd f(e_l) \bigr) f^*\gN(e_k,e_k) f^*\gN(e_l,e_l) \\
			& \ge 2 \sigma \sum_{k=1}^m \frac{2}{1+\lambda_k^2} \frac{\lambda_k^2}{1+\lambda_k^2} \left\{ \sum_{l=1}^m \frac{\lambda_l^2}{1+\lambda_l^2} - \frac{\lambda_k^2}{1+\lambda_k^2} \right \} \,. \qedhere
	\end{align*}
\end{proof}

\begin{corollary}
	\label{cor:trEst}
	Assume $N$ satisfies the curvature bound
	\[
		\sec_N \le -\sigma < 0
	\]
	for some constant $\sigma>0$. If $f_t:\Rb^m\to N$ is a family 
	of strictly length-decreasing maps,
	then under the mean curvature flow, the
	trace of the tensor $\sind$ satisfies
	\[
		(\partial_t - \Delta) \tr(\sind) \ge \frac{\sigma}{2} \bigl(m-\tr(\sind)\bigr)(m-1-\tr(\sind)) \,.
	\]
\end{corollary}

\begin{proof}
	Using $2\tr(f^*\gN) = m - \tr(\sind)$ and $\lambda_k^2\le 1$, we calculate
	\begin{align*}
		(\partial_t - \Delta) \tr(\sind) &\stackrel{\hphantom{{0\le\sind_{kk}\le 1}}}{\ge} 2 \sigma \sum_{k=1}^m \underbrace{\frac{2}{1+\lambda_k^2}}_{\ge 1} \frac{\lambda_k^2}{1+\lambda_k^2} \underbrace{\left\{ \sum_{l=1}^m \frac{\lambda_l^2}{1+\lambda_l^2} - \frac{\lambda_k^2}{1+\lambda_k^2} \right\}}_{\ge 0} \\
			&\stackrel{\hphantom{0\le\sind_{kk}\le 1}}{\ge} 2 \sigma \sum_{k=1}^m \frac{\lambda_k^2}{1+\lambda_k^2} \Biggl\{ \underbrace{\sum_{l=1}^m \frac{\lambda_l^2}{1+\lambda_l^2}}_{= \tr(f^*\gN)} - \frac{\lambda_k^2}{1+\lambda_k^2} \Biggr\} \\
			&\stackrel{\hphantom{0\le\sind_{kk}\le 1}}{=} \sigma \sum_{k=1}^m \underbrace{\frac{\lambda_k^2}{1+\lambda_k^2}}_{=f^*\gN(e_k,e_k)} \Biggl\{ m-\tr(\sind) - \frac{2\lambda_k^2}{1+\lambda_k^2} \Biggr\} \\
			&\stackrel{\hphantom{0\le\sind_{kk}\le 1}}{=} \frac{\sigma}{2} \Biggl\{ \bigl( m - \tr(\sind) \bigr)^2 - 4 \underbrace{\sum_{k=1}^m \left( \frac{\lambda_k^2}{1+\lambda_k^2} \right)^2}_{= \frac{1}{4} \sum_{k=1}^m \left( 1 - \sind_{kk} \right)^2 } \Biggr\} \\
			&\stackrel{\hphantom{0\le\sind_{kk}\le 1}}{=} \frac{\sigma}{2} \left\{ \bigl(m-\tr(\sind)\bigr)^2 - \left( m - 2 \tr(\sind) + \sum_{k=1}^m \sind_{kk}^2 \right) \right\} \\
			&\stackrel{0\le\sind_{kk}\le 1}{\ge}\frac{\sigma}{2} \bigl( (m-\tr(\sind))^2 -( m-\tr(\sind) \bigr) \\
			&\stackrel{\hphantom{0\le\sind_{kk}\le 1}}{=} \frac{\sigma}{2}\bigl(m-\tr(\sind)\bigr)\bigl( m-1-\tr(\sind) \bigr) \,. \qedhere
	\end{align*}
\end{proof}

In the following, we assume $\inf_{x\in\Rb^m} \tr(\sind) < m-1$. Let us set
\[
	\upgamma(x,t) \coloneqq \ee^{\eta t} \phi_R(x) \tr(\sind) - \frac{c_1 (m-1) \exp\left( \frac{\sigma}{2} t \right) - m}{c_1\exp\left( \frac{\sigma}{2}t  \right) - 1} \,,
\]
where
\[
	c_1 \coloneqq 1 + \frac{1}{(m-1)-\inf_{\Rb^m\times\{0\}}\tr(\sind)} > 1 \,.
\]

\begin{lemma}
	\label{lem:gammaEvolEq}
	The function $\upgamma$ satisfies the evolution inequality
	\begin{align*}
		(\partial_t - \Delta) \upgamma &\ge \ee^{\eta t} \phi_R \frac{\sigma}{2} \bigl(m-\tr(\sind)\bigr)\bigl(m-1-\tr(\sind)\bigr) - \frac{\sigma}{2} \frac{c_1 \exp\left(\frac{\sigma}{2} t\right)}{\left( c_1 \exp\left(\frac{\sigma}{2}t\right)-1\right)^2} \\
			& \quad + \ee^{\eta t}\Bigl\{ \eta \phi_R \tr(\sind) - 2 \< \nabla \phi_R, \nabla \tr(\sind) \> - (\Delta \phi_R)\tr(\sind) \Bigr\} \,.
	\end{align*}
\end{lemma}

\begin{proof}
	We calculate
	\begin{align*}
		(\partial_t - \Delta) \upgamma &= \ee^{\eta t}\phi_R (\partial_t - \Delta)\tr(\sind) - \frac{\sigma}{2} \frac{c_1 \exp\left(\frac{\sigma}{2} t\right)}{\left( c_1 \exp\left(\frac{\sigma}{2}t\right)-1\right)^2} \\
			&\quad + \ee^{\eta t}\Bigl\{ \eta \phi_R \tr(\sind) - 2 \< \nabla \phi_R, \nabla \tr(\sind) \> - (\Delta \phi_R)\tr(\sind) \Bigr\} \,.
	\end{align*}
	The claim follows from corollary \ref{cor:trEst}.
\end{proof}

\begin{lemma}
	\label{lem:trDerivEst}
	Let $F(x,t)$ be a smooth solution to \eqref{eq:MCF} with bounded geometry and assume
	there exists $\varepsilon>0$, such that $\sind-\varepsilon\gind\ge 0$ for any $t\in[0,T)$.
	Fix any $T'\in[0,T)$ and $(x_0,t_0)\in\Rb^m\times[0,T']$. Then the following estimate
	holds at $(x_0,t_0)$,
	\[
		- c(T') \frac{\|x_0\|_{\Rb^m}}{R^2} \tr(\sind) \le \< \nabla \phi_R, \nabla \tr(\sind) \> \le c(T') \frac{\|x_0\|_{\Rb^m}}{R^2} \tr(\sind) \,,
	\]
	where $c(T')\ge 0$ is a constant depending only on $T'$.
\end{lemma}

\begin{proof}
	Note that
	\[
		\nabla_u \tr(\sind) = \sum_{k=1}^m (\nabla_u\sind)(e_k,e_k) = 2 \sum_{k=1}^m \RsMN\bigl( \A(u,e_k), \dd F(e_k) \bigr) \,.
	\]
	The bounded geometry assumptions \eqref{eq:BddGeom1} and \eqref{eq:BddGeom2} imply that $\sind$,
	$\nabla\sind$ and therefore $\nabla\tr(\sind)$ are uniformly bounded on $\Rb^m\times[0,T']$ by a constant
	depending only on $T'$. Since also $\tr(\sind)\ge m\varepsilon$ by assumption, 
	at $(x_0,t_0)\in\Rb^m\times[0,T']$ we have
	\[
		-c(T') \frac{\|x_0\|_{\Rb^m}}{R^2} \tr(\sind) \le \< \nabla \phi_R, \nabla \tr(\sind) \> \le c(T') \frac{\|x_0\|_{\Rb^m}}{R^2} \tr(\sind) \,. \qedhere
	\]
\end{proof}

\begin{lemma}
	\label{lem:trTimeEst}
	Assume $m>1$ and that $N$ satisfies the curvature bound
	\[
		\sec_N \le -\sigma < 0
	\]
	for some constant $\sigma>0$. If $f_t:\Rb^m\to N$ is a family of strictly
	length-decreasing maps that evolve under the mean curvature flow
	and $\inf_{\Rb^m\times\{0\}}\tr(\sind)<m-1$,
	then the trace of the tensor $\sind$ satisfies the estimate
	\[
		\tr(\sind) \ge \frac{c_1 (m-1) \exp\left( \frac{\sigma}{2} t \right) - m}{c_1\exp\left( \frac{\sigma}{2}t  \right) - 1}\,,
	\]
	where
	\[
		c_1\coloneqq 1 + \frac{1}{(m-1)-\inf_{\Rb^m\times\{0\}}\tr(\sind)} \,.
	\]
\end{lemma}

\begin{proof}
	We will show that for any fixed $T'\in[0,T)$ and $\eta>0$, there is $R_0>0$ depending only on $\eta$ and $T'$, such that
	$\upgamma>0$ on $\Rb^m\times[0,T']$ for all $R\ge R_0$.
	
	On the contrary, suppose $\upgamma$ is not positive on $\Rb^m\times[0,T']$ for some $R\ge R_0$. Then
	as $\upgamma\ge 0$ on $\Rb^m\times\{0\}$,
	$\tr(\sind)\ge m\varepsilon$ on $\Rb^m\times[0,T)$ and $\phi_R(x)\to\infty$
	as $\|x\|\to\infty$, it follows that $\upgamma>0$ outside some compact set $K\subset\Rb^m$ for
	all $t\in[0,T)$. We conclude that there is $(x_0,t_0)\in K\times[0,T']$ such that $\upgamma(x_0,t_0)=0$
	and that $t_0$ is the first such time. According to the second derivative criterion, at the point
	$(x_0,t_0)$ we have
	\begin{equation}
		\label{eq:gammaSDC}
		\upgamma = 0 \,, \qquad \partial_t \upgamma \le 0\,, \qquad \nabla \upgamma = 0 \qquad \text{and} \qquad \Delta \upgamma \ge 0 \,.
	\end{equation}
	On the other hand, using lemma \ref{lem:gammaEvolEq}, we estimate
	the terms in the evolution equation for $\gamma$,
	\begin{align*}
		(\partial_t - \Delta) \upgamma &\ge \Ac + \Bc \,,
	\end{align*}
	where
	\begin{align*}
		\Ac &\coloneqq \ee^{\eta t}\phi_R \frac{\sigma}{2} \bigl(m-\tr(\sind)\bigr) \bigl(m-1-\tr(\sind)\bigr) - \frac{\sigma}{2} \frac{c_1 \exp\left(\frac{\sigma}{2} t\right)}{\left( c_1 \exp\left(\frac{\sigma}{2}t\right)-1\right)^2}  \,, \\
		\Bc &\coloneqq \ee^{\eta t}\Bigl\{ \eta \phi_R \tr(\sind) - 2 \< \nabla \phi_R, \nabla \tr(\sind) \> - (\Delta \phi_R)\tr(\sind) \Bigr\} \,.
	\end{align*}
	Further, using Eq.\ \eqref{eq:gammaSDC}, at $(x_0,t_0)$ we obtain
	\[
		\ee^{\eta t_0} \phi_R \tr(\sind) = \frac{c_1 (m-1) \exp\left(\frac{\sigma}{2} t_0\right)-m}{c_1 \exp\left(\frac{\sigma}{2}t_0\right)-1} \,.
	\]
	Since $c_1>1$, 
	the right-hand side is monotonically increasing in $t$ for $t\ge 0$ and
	bounded from above by $m-1$, so that
	\[
		\tr(\sind) \le \ee^{\eta t_0} \phi_R \tr(\sind) = \frac{c_1 (m-1) \exp\left(\frac{\sigma}{2} t_0\right)-m}{c_1 \exp\left(\frac{\sigma}{2}t_0\right)-1} \le m-1 \,.
	\]
	Consequently, we estimate
	\begin{align*}
		\Ac &\ge \frac{\sigma}{2}\left( m - \frac{c_1(m-1)\exp\left(\frac{\sigma}{2}t_0\right)-m}{c_1\exp\left(\frac{\sigma}{2}t_0\right)-1} \right)\bigl(m-1-\tr(\sind)\bigr) \\
			& \qquad\qquad\qquad\qquad\qquad\qquad\qquad\qquad\qquad - \frac{\sigma}{2} \frac{c_1\exp\left(\frac{\sigma}{2}t_0\right)}{\left(c_1\exp\left(\frac{\sigma}{2}t_0\right)-1\right)^2} \\
			&\ge \frac{\sigma}{2} \frac{c_1 \exp\left(\frac{\sigma}{2}t_0\right)}{c_1\exp\left(\frac{\sigma}{2}t_0\right)-1}\left(m-1-\frac{c_1(m-1)\exp\left(\frac{\sigma}{2}t_0\right)-m}{c_1\exp\left(\frac{\sigma}{2}t_0\right)-1}  \right) \\
			& \qquad\qquad\qquad\qquad\qquad\qquad\qquad\qquad\qquad - \frac{\sigma}{2} \frac{c_1\exp\left(\frac{\sigma}{2}t_0\right)}{\left(c_1\exp\left(\frac{\sigma}{2}t_0\right)-1\right)^2} \\
			&= \frac{\sigma}{2}  \frac{c_1 \exp\left(\frac{\sigma}{2}t_0\right)}{c_1\exp\left(\frac{\sigma}{2}t_0\right)-1}  \frac{1}{c_1\exp\left(\frac{\sigma}{2}t_0\right)-1} - \frac{\sigma}{2} \frac{c_1\exp\left(\frac{\sigma}{2}t_0\right)}{\left(c_1\exp\left(\frac{\sigma}{2}t_0\right)-1\right)^2} \\
			&= 0 \,.
	\end{align*}
	For the terms in $\Bc$, we use lemma \ref{lem:trDerivEst} to calculate
	\[
		\Bc \ge \ee^{\eta t_0} \left\{ \eta + \eta\frac{\|x_0\|^2_{\Rb^m}}{R^2} - 3 c(T') \frac{\|x_0\|_{\Rb^m}}{R^2} - \frac{c(T')}{R^2} \right\} \tr(\sind) \,.
	\]
	Now choosing $R_0>0$ (depending on $\eta$ and $T'$) large enough, the term
	\[
		\frac{\eta}{2} + \eta \frac{\|x_0\|^2}{R^2} - 3 c(T') \frac{\|x_0\|_{\Rb^m}}{R^2} - \frac{c(T')}{R^2}
	\]
	is strictly positive for any $R\ge R_0$ and any $\|x_0\|_{\Rb^m }$. This yields
	\[
		(\partial_t - \Delta) \upgamma(x_0,t_0) = \Ac + \Bc > \ee^{\eta t_0} \frac{\eta}{2} \tr(\sind) \ge \ee^{\eta t_0} \frac{\eta}{2} m\varepsilon > 0 \,,
	\]
	which contradicts \eqref{eq:gammaSDC} and thus shows the claim.

	The statement of the lemma follows by first letting $R\to\infty$, then $\eta\to 0$ and finally $T'\to T$.
\end{proof}

\begin{remark}
	We observe that the estimate in lemma \ref{lem:trEst} holds for any $\sigma\ge 0$, i.\,e.\
	\[
		(\partial_t - \Delta) \tr(\sind) \ge 0 \,.
	\]
	Arguing in the same way as in the proof of lemma \ref{lem:trTimeEst}, we obtain that
	$\tr(\sind)$ is non-decreasing in time.
\end{remark}

%
%
\section{Higher-Order Estimates}

In this section we prove decay estimates for derivatives
of order $n\ge 2$ of the function $f_t$ defining the graph. 
We begin by recalling some definitions, where we follow
\cite[Section 5]{SHS16} (see also \cite{CH10}).

\begin{definition}[{$C^{\infty}$-convergence}]
	Let $(E,\pi,M)$ be a vector bundle endowed with a Riemannian metric $\gind$ and a metric
	connection $\nabla$ and suppose that $\{\xi_k\}_{k\in\Nb}$ is a sequence of sections of $E$.
	Let $U$ be an open subset of $M$ with compact closure $\overline{U}$ in $M$. Fix a natural number
	$p\ge 0$. We say that $\{\xi_k\}_{k\in\Nb}$ \emph{converges in $C^p$} to $\xi_{\infty}\in \Gamma(E|_{\overline{U}})$,
	if for every $\varepsilon>0$ there exists $k_0=k_0(\varepsilon)$, such that
	\[
		\sup_{0\le\alpha\le p} \sup_{x\in\overline{U}} \left| \nabla^{\alpha}(\xi_k - \xi_{\infty}) \right| < \varepsilon
	\]
	whenever $k\ge k_0$. We say that $\{\xi_k\}_{k\in\Nb}$ converges in $C^{\infty}$ to $\xi_{\infty} \in \Gamma(E|_{\overline{U}})$
	if $\{\xi_k\}_{k\in\Nb}$ converges in $C^p$ to $\xi_{\infty}\in\Gamma(E|_{\overline{U}})$ for any $p\ge 0$.
\end{definition}

\begin{definition}[{$C^{\infty}$-convergence on compact sets}]
	Let $(E,\pi,M)$ be a vector bundle endowed with a Riemannian metric $\gind$ and a metric
	connection $\nabla$. Let $\{U_n\}_{n\in\Nb}$ be an exhaustion of $M$ and $\{\xi_k\}_{k\in\Nb}$
	be a sequence of sections of $E$ defined on open sets $A_k$ of $M$. We say that $\{\xi_k\}_{k\in\Nb}$
	\emph{converges smoothly on compact sets} to $\xi_{\infty}\in\Gamma(E)$ if:
	\begin{enumerate}[label=(\roman*)]
		\item For every $n\in\Nb$ there exists $k_0$ such that $\overline{U}_n\subset A_k$
			for all natural numbers $k\ge k_0$.
		\item The sequence $\{\xi|_{\overline{U}}\}_{k\ge k_0}$ converges in $C^{\infty}$ to the restriction
			of the section $\xi_{\infty}$ on $\overline{U}_n$.
	\end{enumerate}
\end{definition}

\begin{definition}[Pointed Manifold]
	A \emph{pointed Riemannian manifold} $(M,\gind,x)$ is a Riemannian manifold
	$(M,\gind)$ with a choice of base point $x\in M$. If the metric $\gind$ is
	complete, we say that $(M,\gind,x)$ is a \emph{complete pointed Riemannian manifold}.
\end{definition}

\begin{definition}[Cheeger-Gromov smooth convergence]
	A sequence of complete pointed Riemannian manifolds $\{(M_k,\gind_k,x_k)\}_{k\in\Nb}$
	smoothly converges in the sense of Cheeger-Gromov to a complete pointed Riemannian
	manifold $(M_{\infty},\gind_{\infty},x_{\infty})$, if there exists:
	\begin{enumerate}[label=(\roman*)]
		\item An exhaustion $\{U_k\}_{k\in\Nb}$ of $M_{\infty}$ with $x_{\infty}\in U_k$, for all $k\in\Nb$.
		\item A sequence of diffeomorphisms $\Phi_k:U_k\to\Phi_k(U_k)\subset M_k$ with $\Phi_k(x_{\infty})=x_k$
		and such that $\{\Phi_k^*\gind_k\}_{k\in\Nb}$ smoothly converges in $C^{\infty}$ to $\gind_{\infty}$
		on compact sets in $M_{\infty}$.
	\end{enumerate}
	The family $\{(U_k,\Phi_k)\}_{k\in\Nb}$ is called a \emph{family of convergence pairs}
	of the sequence $\{(M_k,\gind_k,x_k)\}_{k\in\Nb}$ with respect to the limit $(M_{\infty},\gind_{\infty},x_{\infty})$.
\end{definition}

In the sequel, when we say \emph{smooth convergence}, we will always mean smooth convergence
in the sense of Cheeger-Gromov. \\

The family of convergence pairs is not unique. However, two families of convergence pairs 
$\{(U_k,\Phi_k)\}_{k\in\Nb}$ and
$\{(W_k,\Psi_k)\}_{k\in\Nb}$ are equivalent in the sense that there exists an isometry
$\mathcal{I}$ of the limit $(M_{\infty},\gind_{\infty},x_{\infty})$ such that,
for every compact subset $K$ of $M_{\infty}$ there exists a natural number $k_0$ such that
for any $k\ge k_0$:
\begin{enumerate}[label=(\roman*)]
	\item the mapping $\Phi_k^{-1}\circ \Psi_k$ is well-defined over $K$ and 
	\item the sequence $\{\Phi_k^{-1}\circ \Psi_k\}_{k\ge k_0}$ smoothly converges to $\mathcal{I}$ on $K$.
\end{enumerate}
The limiting pointed Riemannian manifold $(M_{\infty}, \gind_{\infty}, x_{\infty})$
is unique up to isometries \cite[Lemma 5.5]{MT07}. \\

The following proposition is quite standard.
\begin{proposition}
	Let $(M,\gind)$ be a complete Riemannian manifold with bounded geometry. Suppose that $\{a_k\}_{k\in\Nb}$
	is an increasing sequence of real numbers that tends to $\infty$ and let $\{x_k\}_{k\in\Nb}$
	be a sequence of points on $M$. Then the sequence $(M,a_k^2\gind, x_k)$ smoothly subconverges
	to the standard Euclidean space $(\Rb^m,\gind_{\mathrm{eucl.}},0)$.
\end{proposition}

\begin{definition}
	We say the a sequence $\{(M_k,\gind_k,x_k)\}_{k\in\Nb}$ of complete pointed
	Riemannian manifolds has uniformly bounded geometry if the following conditions are
	satisfied:
	\begin{enumerate}[label=(\roman*)]
		\item For any $j\ge 0$ there exists a uniform constant $C_j\ge 0$, such that for each
			$k\in\Nb$ it holds $\|\nabla^j \Rind_{M_k}\| \le C_j$.
		\item There exists a uniform constant $c_0$ such that $\mathrm{inj}(M_k)\ge c_0>0$.
	\end{enumerate}
\end{definition}

\begin{theorem}[Cheeger-Gromov compactness]
	Let $\{(M_k,\gind_k,x_k)\}_{k\in\Nb}$ be a sequence of complete pointed Riemannian manifolds
	with uniformly bounded geometry. Then the sequence $\{(M_k,\gind_k,x_k)\}_{k\in\Nb}$
	subconverges smoothly to a complete pointed Riemannian manifold $(M_{\infty},\gind_{\infty},x_{\infty})$.
\end{theorem}

\begin{definition}[Convergence of isometric immersions]
	Let $F_k:(M_k,\gind_k,x_k) \to (L_k,\h_k,y_k)$ be a sequence of isometric
	immersions, such that $F_k(x_k)=y_k$ for any $k\in\Nb$.
	We say
	that the sequence $\{F_k\}_{k\in\Nb}$ converges smoothly to an isometric immersion
	$F_{\infty}:(M_{\infty},\gind_{\infty},x_{\infty})\to(L_{\infty},\h_{\infty},y_{\infty})$ if the
	following conditions are satisfied:
	\begin{enumerate}[label=(\roman*)]
		\item The sequence $\{(M_k,\gind_k,x_k)\}_{k\in\Nb}$ smoothly converges to $(M_{\infty},\gind_{\infty},x_{\infty})$.
		\item The sequence $\{(L_k,\h_k,y_k)\}_{k\in\Nb}$ smoothly converges to $(L_{\infty},\h_{\infty}, y_{\infty})$.
		\item If $\{(U_k,\Phi_k)\}_{k\in\Nb}$ is a family of convergence pairs of 
			$\{(M_k,\gind_k,x_k)\}_{k\in\Nb}$ and $\{(W_k,\Psi_k)\}_{k\in\Nb}$ is a family
			of convergence pairs of $\{(L_k,\h_k,y_k)\}_{k\in\Nb}$, then
			for each $k\in\Nb$ the relation $F_k\circ\Phi_k(U_k)\subset\Psi_k(W_k)$ holds and
			$\Psi^{-1}_k\circ F\circ \Phi_k$ smoothly converges to $F_{\infty}$ on compact sets.
	\end{enumerate}
\end{definition}

Assume that $(M,\gM)$ and $(N,\gN)$ are complete manifolds with bounded geometry and 
that $F:M\times[0,T)\to N$ is a solution of the mean curvature flow with bounded geometry.
For any $\tau>0$ and $(x_0,t_0)\in M\times[0,T)$, let us define the 
\emph{parabolic scaling by $\tau$ at $(x_0,t_0)$} by letting
\begin{align*}
	F^{\tau}_t : (M,x_0) \to \bigl( N, \tau^2 \gN, F^{\tau}_t(x_0) \bigr) \,, \qquad F^{\tau}_t(x) \coloneqq F_{t_0+t/\tau^2}\bigl( x \bigr) \,.
\end{align*}

\begin{remark}
	The parabolic scaling preserves the graph property. More precisely,
	assume that the family of immersions $F:M\times[0,T)\to M\times N$
	are graphs with induced metric metric $\gind = F^*\gMN$. Then there is
	a family of diffeomorphisms $\varphi:M\times[0,T)\to M$ and a function $f_t:M\to N$, 
	such that $F_t\circ\varphi_t(x) = \bigl( x, f_t(x) \bigr)$.
	Let the parabolic scaling by $\tau$ at $(x_0,t_0)$
	be given by $F^{\tau}_t(x) = F_{t_0+t/\tau^2}(x)$, and let $\varphi^{\tau}_t(x) \coloneqq \varphi_{t_0+t/\tau^2}(x)$.
	We calculate
	\[
		F^{\tau}_{t}\circ\varphi^{\tau}_t(x) = F_{t_0+t/\tau^2}\bigl( \varphi_{t_0+t/\tau^2}(x) \bigr) = \bigl( x, f_{t_0+t/\tau^2}(x) \bigr) \,.
	\]
	Furthermore, the metric induced on the graph is given by
	\[
			\gind_{\tau} \coloneqq (F_t^{\tau})^*\gMN = \tau^2 \gind \,.
	\]
	We note that a length-decreasing property of $f_t$ is preserved under the scaling.
\end{remark}

\begin{lemma}
	\label{lem:HOEst1}
	Let $F:\Rb^m\times[0,T)\to\Rb^m\times N$ be a smooth, graphic solution
	to \eqref{eq:MCF} with bounded geometry. Suppose the corresponding maps
	$f_t:\Rb^m\to N$ satisfy $\|\dd f_t\|^2\le C_1$ and $\|\nabla \dd f_t\|^2\le C_2$
	on $\Rb^m\times[0,T)$ for some constants $C_1,C_2\ge 0$. Then for every $l\ge 3$
	there exists a constant $C_l$, such that
	\[
		\sup_{x\in\Rb^m} \|\nabla^{l-1} \dd f_t \|^2 \le C_l
	\]
	for all $t\in[0,T)$.
\end{lemma}

\begin{proof}
	If $\|\nabla^2 \dd f_t\|^2\le C_3$ in $[0,T)$, then a parabolic bootstrapping
	argument for the quasilinear Eq.\ \eqref{eq:pMCF} gives $\|\nabla^{l-1} \dd f_t\|^2\le C_l$ for $l\ge 2$.
	It is therefore sufficient to prove the claim for $l=3$.
	
	Suppose $\|\nabla^2 \dd f_t\|^2$ was not bounded on $\Rb^m\times[0,T)$. By the bounded
	geometry assumption on $F$ (resp.\ on $f$), there would be a sequence $t_k\to T$, such that
	\[
		2 \mu_k \coloneqq \sup_{x\in\Rb^m} \| \nabla^2 \dd f_{t_k}(x) \|^2 \to \infty \quad \text{and} \quad \sup_{\substack{x\in\Rb^m\\t\le t_k}} \|\nabla^2 \dd f_t(x)\|^2 \le 2\mu_k < \infty\,.
	\]
	This implies there is a sequence $\{x_k\}$, such that $\|\nabla^2 \dd f_{t_k}(x_k)\|^2 \ge \mu_k \to \infty$
	for $t_k\to T$. Set $\tau_k \coloneqq \mu_k^{1/4}$ and consider the sequence
	\[
		F^{\tau_k}_k : \bigl( \Rb^m, F^{\tau_k\,*}_k(\tau_k^2\RgMN), x_k \bigr) \to \bigl( \Rb^m\times N, \tau_k^2 \RgMN, F^{\tau_k}_k(x_k) \bigr)
	\]
	of immersions of pointed Riemannian manifolds, where
	\[
		F^{\tau_k}_k(x,\tScaled) \coloneqq F\left( x, \frac{\tScaled}{\tau_k^2} + t_k \right) \,.
	\]
	Then $\{F^{\tau_k}_k\}$ is a sequence of mean curvature flows for $\tScaled\in[-\tau_k^2t_k,0]$.
	
	Since $N$ has bounded geometry, $\bigl( \Rb^m\times N, \tau_k^2 \RgMN, F^{\tau_k}_k(x_k) \bigr)$
	converges on every compact subset
	to the Euclidean space $\bigl(\Rb^{m+n},\gind_{\Rb^m\times\Rb^n},0)$ with its standard metric.
	The sequence $\bigl(\Rb^m,F^{\tau_k\,*}_k(\tau_k^2\RgMN),x_k\bigr)$ converges
	smoothly to a geometric limit $(M_{\infty}, \gind_{\infty}, x_{\infty})$ in the Cheeger-Gromov sense,
	and every manifold of the sequence 
	is the graph of a function $\fScaled{k}$. In particular, the	
	limiting manifold $M_{\infty}$ is the graph of a function $\fScaled{\infty}:\Rb^m\to \Rb^n$.
	Since $M_{\infty}$ satisfies the mean curvature flow equation in $\Rb^{m+n}$, 
	the function $\fScaled{\infty}$ satisfies the quasilinear parabolic equation
	\[
		\partial_t f^{\infty}_t(x) = \sum_{i,j=1}^m \gind^{ij}_{\infty} \frac{\partial^2 f^{\infty}}{\partial x^i \partial x^j} \,.
	\]
	Using the assumptions $\|\dd f_t\|^2 \le C_1$ and $\| \nabla \dd f_t\|^2\le C_2$
	in $\Rb^m\times[-\tau_k^2t_k,0]$, we calculate 
	\begin{align*}
		\| \dd \fScaled{\tau_k}(x,\tScaled) \|^2_{\tau_k} &= \| \dd f(x,t) \|^2 \le C_1 \,, \\
		\| \nabla \dd \fScaled{\tau_k}(x,\tScaled) \|^2_{\tau_k} &= \frac{\| \nabla \dd f(x,t)\|^2}{\tau_k^2} \le \frac{C_2}{\tau_k^2} \quad \xrightarrow{k\to\infty} \quad 0 \,,
	\end{align*}
	where $\|\cdot\|_{\tau_k}$ denotes the norms with respect to the rescaled metrics $\tau_k^2\gN$
	and $\tgind_{\tau_k}$.
	Using the definition of $\tau_k$ and the definition of the sequence $(x_k,t_k)$, we obtain
	\[
		\bigl\| \nabla^2 \dd \fScaled{\tau_k}(x,\tScaled) \bigr\|^2_{\tau_k} = \frac{\| \nabla^2 \dd f(x,t) \|^2}{\tau_k^4} = \frac{\| \nabla^2 \dd f(x,t) \|^2}{\mu_k} \le 2
	\]
	and
	\[
		\bigl\| \nabla^2 \dd \fScaled{\tau_k}(x_k,0) \bigr\|^2_{\tau_k} = \frac{ \|\nabla^2 \dd f(x_k,t_k) \|^2}{\tau_k^4} = \frac{\|\nabla^2 \dd f(x_k,t_k)\|^2}{\mu_k} \ge 1 \,.
	\]
	Since  the sub-convergence to $(M_{\infty},\gind_{\infty},x_{\infty})$ was smooth,
	the limit $\fScaled{\infty}$ is smooth and satisfies
	\[
		\| \nabla \dd \fScaled{\infty}\| = 0 \qquad \text{and} \qquad \bigl\| \nabla^2 \dd \fScaled{\infty} (x_{\infty},0)\| \ge 1 \,,
	\]
	which is a contradiction, so that $\|\nabla^2 \dd f_t\|$ has to be bounded.
	
	The estimates for the higher-order derivatives follow by differentiating
	equation \eqref{eq:pMCF} and repeating the above argument.
	We thus obtain that the spatial derivatives as well as the time derivatives
	of $f_t$ of any positive order are uniformly bounded.
\end{proof}

\begin{lemma}
	\label{lem:HOBd}
	Let $F:\Rb^m\times[0,T)\to \Rb^m\times N$ be a smooth, graphic solution
	to \eqref{eq:MCF} with bounded geometry and denote by $f_t:\Rb^m\to N$
	the corresponding maps. Assume $f_0$ is strictly length-decreasing
	and
	further assume that $\|\Hv\|\le C$ on $\Rb^m\times[0,T)$ for some
	constant $C\ge 0$. Then for every $k\ge 1$ there exists a constant
	$C_k\ge 0$, such that
	\[
		\sup_{x\in\Rb^m} \| \nabla^k \dd f_t(x)\|^2 \le C_k
	\]
	for all $t\in[0,T)$.
\end{lemma}

\begin{proof}
	By lemma \ref{lem:LenPres}, the length-decreasing condition is preserved in $[0,T)$,
	so that the relation $f_t^*\gN\le(1-\delta)\RgM$ holds in $[0,T)$. This shows the claim
	for $l=1$. By lemma \ref{lem:HOEst1}, we only need to prove the case $l=2$. Suppose
	the claim was false for $l=2$. Let
	\[
		\eta(t) \coloneqq \sup_{\substack{x\in\Rb^m\\t'\le t}} \| \nabla \dd f(x,t') \| \,.
	\]
	Then there is a sequence $(x_k,t_k)$ along which we have $\|\nabla \dd f(x_k,t_k)\|\ge \eta(t_k)/2$
	while $\eta(t_k) \to\infty$ as $t_k\to T$. Let $\tau_k \coloneqq \eta(t_k)$. 
	Consider the sequence
	\[
		F^{\tau_k}_k : \bigl( \Rb^m, F^{\tau_k\,*}_k(\tau_k^2\RgMN), x_k \bigr) \to \bigl( \Rb^m\times N, \tau_k^2 \RgMN, F^{\tau_k}_k(x_k) \bigr)
	\]
	of immersions of pointed Riemannian manifolds, where
	\[
		F^{\tau_k}_k(x,\tScaled) \coloneqq F\left( x, \frac{\tScaled}{\tau_k^2} + t_k \right) \,.
	\]
	Then $\{F^{\tau_k}_k\}$ is a sequence of mean curvature flows for $s\in[-\tau_k^2t_k,0]$.
	
	Since $N$ has bounded geometry, $\bigl( \Rb^m\times N, \tau_k^2 \RgMN, F^{\tau_k}_k(x_k) \bigr)$
	converges on every compact subset
	to the Euclidean space $\bigl(\Rb^{m+n},\gind_{\Rb^m\times\Rb^n},0)$ with its standard metric.
	The sequence $\bigl(\Rb^m,F^{\tau_k\,*}_k(\tau_k^2\RgMN),x_k\bigr)$ converges
	smoothly to a geometric limit $(M_{\infty}, \gind_{\infty}, x_{\infty})$ in the Cheeger-Gromov sense,
	and every manifold of the sequence is the graph of a function $\fScaled{k}$.
	In particular, the	
	limiting manifold $M_{\infty}$ is the graph of a function $\fScaled{\infty}:\Rb^m\to \Rb^n$. Since
	$M_{\infty}$ satisfies the mean curvature flow equation, the function $\fScaled{\infty}$ satisfies
	the quasilinear parabolic equation
	\[
		\partial_t \fScaled{\infty} = \sum_{i,j=1}^m \tgind_{\infty}^{ij} \frac{\partial^2 \fScaled{\infty}}{\partial x^i \partial x^j} \,.
	\]
	Note that by the definition of $\tau_k=\eta(t_k)$, it is
	\begin{align*}
		\| \dd \fScaled{\tau_k}(x,\tScaled) \|_{\tau_k} &= \| \dd f(x,t) \| \le C_1 \,, \\
		\| \nabla \dd \fScaled{\tau_k}(x,\tScaled) \|_{\tau_k} &= \tau_k^{-1} \| \nabla \dd f(x,t)\|^2 \le 1
	\end{align*}
	for all $(x,\tScaled)\in\Rb^m\times[-\tau_k^2 t_k,0]$,
	where $\|\cdot\|_{\tau_k}$ denotes the norms with respect to the rescaled metrics $\tau_k^2\gN$
	and $\tgind_{\tau_k}$.
	Moreover, by the definition
	of the sequence $(x_k,t_k)$, the estimate
	\begin{equation}
		\label{eq:HOEst}
		\|\nabla \dd \fScaled{\tau_k}(x_k,0)\|_{\tau_k} = \frac{\|\nabla \dd f(x_k,t_k)\|}{\tau_k} = \frac{\|\nabla \dd f(x_k,t_k)\|}{\eta(t_k)} \ge \frac{1}{2}
	\end{equation}
	holds. By lemma \ref{lem:HOEst1} we conclude that all higher-order derivatives of $\fScaled{\tau_k}$
	are uniformly bounded on $\Rb^m\times[-\tau_k^2t_k,0]$. 
	Since $\|\Hv\|\le C$ for the graphs $(x,f(x,t))$ by assumption, after rescaling we have
	\[
		\|\Hv_{\tau_k}\|_{\tau_k} \le \frac{C}{\tau_k}
	\]
	for the graphs $(x,\fScaled{\tau_k}(x,\tScaled))$. It follows that for each $\tScaled$,
	the limiting graph 
	\[
		\Gamma\bigl(\fScaled{\infty}(\cdot,r)\bigr) \subset \Rb^m\times\Rb^n
	\]
	must have $\Hv_{\infty}=0$ everywhere,
	as well as $\widetilde{\lambda}_i^2 \coloneqq \fScaled{\infty}^*\RgM(e_i,e_i)\le 1-\delta$. This in
	turn implies bounds on the Jacobian of the projection $\pi_1$ from the graph $(x,\fScaled{\infty}(x,r))$ to $\Rb^m$,
	\[
		\frac{1}{2^{m/2}} < \star{}\Omega_{\infty} = \frac{1}{\sqrt{\prod_{i=1}^m\bigl( 1+\widetilde{\lambda}_i^2 \bigr)}} \le 1 \,.
	\]
	Thus, we can apply a Bernstein-type theorem of Wang
	\cite{Wan03a}*{Theorem 1.1}
	to conclude that the graph $(x,\fScaled{\infty}(x,\tScaled))$ is an affine subspace of $\Rb^m\times\Rb^n$.
	Therefore, $\fScaled{\infty}$ has to be a linear map for each $\tScaled$, 
	but this contradicts \eqref{eq:HOEst}, which
	(taking the limit $k\to\infty)$ implies the estimate $\|\D^2\fScaled{\infty}(x_{\infty},0)\|\ge 1/2$.
\end{proof}

\begin{lemma}
	Let $f_0:\Rb^m\to N$ be smooth and length-decreasing. 
	Then Eq.\ \eqref{eq:pMCF} has a smooth solution $f_t(x)$ on $\Rb^m\times[0,\infty)$ with initial
	condition $f_0(x)$, such that
	for all $k\ge 2$ we have the estimate
	\[
		t^{k-1} \sup_{x\in\Rb^m} \| \nabla^{k-1} \dd f_t(x) \| \le C_{k,\delta}
	\]
	for some constant $C_{k,\delta}\ge 0$ depending only on $k$ and $\delta$.
\end{lemma}

\begin{proof}
	By theorem \ref{thm:ShortTimeEx} we know that \eqref{eq:pMCF}
	has a smooth short-time solution on $[0,T)$ with initial condition
	$f_0(x)$ and bounded geometry for any $t\in[0,T)$. 
	Assume $T<\infty$. 
	Then lemma \ref{lem:HOBd} implies $\|\nabla^{k-1}\dd f\|\le C_{k,\delta}$ 
	in $[0,T)$ for all integers $k\ge 1$, which by continuity extends
	to $[0,T]$. By theorem \ref{thm:ShortTimeEx}, we can extend the solution
	beyond $T$, which contradicts the definition of $T$, so that $T=\infty$.
	
	For the second part, first consider $k=1$, i.\,e. we show
	\begin{equation}
		\label{eq:decay2}
		\sup_{x\in\Rb^m} \|\nabla \dd f(x,t)\|^2 t \le C_{1,\delta}
	\end{equation}
	for any $t$. Assume this is not the case. Since $f(x,t)$ satisfies
	the bounded geometry condition, it is
	\[
		\sup_{x\in\Rb^m} \|\nabla \dd f(x,t)\| < \infty
	\]
	for any $t$. If Eq.\ \eqref{eq:decay2} does not hold, there exists
	a sequence $t_k\to\infty$ with
	\begin{equation}
		\label{eq:HORescale}
		\sup_{\substack{x\in\Rb^m\\ t\le t_k}} \|\nabla \dd f(x,t)\|^2 t_k = \sup_{x\in\Rb^m} \| \nabla \dd f(x,t_k) \|^2 t_k \eqqcolon 2\mu_k \quad \xrightarrow{k\to\infty} \infty \,.
	\end{equation}
	Further, for each $t_k$, there exists $x_k\in\Rb^m$, such that
	\begin{equation}
		\label{eq:xkDef}
		\|\nabla \dd f(x_k,t_k)\|^2 t_k \ge \mu_k \,.
	\end{equation}
	Let $\tau_k\coloneqq\sqrt{\mu_k/t_k}$ and consider the sequence
	\[
		F^{\tau_k}_k : \bigl( \Rb^m, F^{\tau_k\,*}_k(\tau_k^2\RgMN), x_k \bigr) \to \bigl( \Rb^m\times N, \tau_k^2 \RgMN, F^{\tau_k}_k(x_k) \bigr)
	\]
	of immersions of pointed Riemannian manifolds, where
	\[
		F^{\tau_k}_k(x,\tScaled) \coloneqq F\left( x, \frac{\tScaled}{\tau_k^2} + t_k \right) \,.
	\]
	Then $\{F^{\tau_k}_k\}$ is a sequence of mean curvature flows for $s\in[-\tau_k^2t_k,0]$.
	
	Since $N$ has bounded geometry, $\bigl( \Rb^m\times N, \tau_k^2 \RgMN, F^{\tau_k}_k(x_k) \bigr)$
	converges on every compact subset
	to the Euclidean space $\bigl(\Rb^{m+n},\gind_{\Rb^m\times\Rb^n},0)$ with its standard metric.
	The sequence $\bigl(\Rb^m,F^{\tau_k\,*}_k(\tau_k^2\RgMN),x_k\bigr)$ converges
	smoothly to a geometric limit $(M_{\infty}, \gind_{\infty}, x_{\infty})$ in the Cheeger-Gromov sense,
	and every manifold of the sequence is the graph of a function $\fScaled{k}$.
	In particular, the	
	limiting manifold $M_{\infty}$ is the graph of a function $\fScaled{\infty}:\Rb^m\to \Rb^n$. Since
	$M_{\infty}$ satisfies the mean curvature flow equation, the function $\fScaled{\infty}$ satisfies
	the quasilinear parabolic equation
	\[
		\partial_t \fScaled{\infty} = \sum_{i,j=1}^m \tgind_{\infty}^{ij} \frac{\partial^2 \fScaled{\infty}}{\partial x^i \partial x^j} \,.
	\]
	Further, by the definition of $\tau_k$ and $(x_k,t_k)$, the relations
	\[
		\|\dd \fScaled{\tau_k}(x,r)\|_{\tau_k} = \|\dd f(x,t)\| \,, \qquad \bigl\| \nabla \dd \fScaled{\tau_k}(x,r)\bigr\|_{\tau_k} = \frac{\|\nabla \dd f(x,t)\|}{\sqrt{\mu_k/t_k}} \stackrel{\text{Eq.\ \eqref{eq:HORescale}}}{\le}\sqrt{2}
	\]
	hold on $\Rb^m\times[-\mu_k,0]$. On the other hand, by the definition
	of $\mu_k$, it is
	\begin{equation}
		\label{eq:Lim2ndDeriv}
		\|\nabla \dd \fScaled{\infty}(x_k,0)\|_{\tau_k} = \frac{1}{\mu_k/t_k} \|\nabla \dd f(x_k,t_k)\| \stackrel{\text{Eq.\ \eqref{eq:xkDef}}}{\ge} 1\,.
	\end{equation}
	Note that for the graph of $f(x,t)$, we have $\|\Hv\|\le C$
	for all $t$ and some constant $C\ge 0$. Thus, for any
	$0<\mu<\mu_k$ we conclude
	\[
		\|\Hv_{\tau_k}\|^2_{\tau_k} \le \frac{C}{\mu_k} \quad \xrightarrow{k\to\infty} 0
	\]
	uniformly on $\Rb^m\times[-\mu,0]$. It follows that $\Hv_{\infty}=0$,
	so that $\bigl(y,\fScaled{\infty}(y,r)\bigr)$ is a minimal graph
	in $\Rb^n\times\Rb^m$. 
	Since
	\[
		\fScaled{\tau_k}^*(\tau_k^2\gN)(v,v) = \tau_k^2 \fScaled{\tau_k}^*\gN(v,v) \le (1-\delta) \tau_k^2 \gM(v,v) \,,
	\]
	the maps $\fScaled{\tau_k}$ are strictly length-decreasing, so that in particular
	the limiting map $\fScaled{\infty}:\Rb^m\to\Rb^n$ is strictly length-decreasing.
	Thus, by applying the
	Bernstein-type theorem \cite{Wan03a}*{Theorem 1.1}, we conclude
	that $\fScaled{\infty}$ has to be an affine map.
	But this contradicts Eq.\ \eqref{eq:Lim2ndDeriv}, so that
	our initial assumption was false, which eventually proves
	Eq.\ \eqref{eq:decay2}.
	
	Now let $l\ge 3$ and suppose that $\|\nabla^{l-1} \dd f\|^2 t^{l-1}$
	is not uniformly bounded. Then there exists a sequence $(x_k,t_k)$,
	such that
	\[
		\sup_{\substack{x\in\Rb^m\\t\le t_k}} \| \nabla^{l-1} \dd f(x,t)\|^2 t^{l-1} \eqqcolon 2\sigma_k \qquad \xrightarrow{k\to\infty} \infty
	\]
	and
	\begin{equation}
		\label{eq:xkDef2}
		\| \nabla^{l-1} \dd f(x_k,t_k) \|^2 t_k^{l-1} \ge \sigma_k \,.
	\end{equation}
	Let 
	\[
		\tau_k \coloneqq \sqrt{\frac{\sigma_k^{1/(l-1)}}{t_k}}
	\]
	and consider the sequence
	\[
		F^{\tau_k}_k : \bigl( \Rb^m, F^{\tau_k\,*}_k(\tau_k^2\RgMN), x_k \bigr) \to \bigl( \Rb^m\times N, \tau_k^2 \RgMN, F^{\tau_k}_k(x_k) \bigr)
	\]
	of immersions of pointed Riemannian manifolds
	for $t_k/2\le t\le t_k$, where
	\[
		F^{\tau_k}_k(x,\tScaled) \coloneqq F\left( x, \frac{\tScaled}{\tau_k^2} + t_k \right) \,.
	\]
	Then $\{F^{\tau_k}_k\}$ is a sequence of mean curvature flows for $r\in[-\sigma_k^{1/(l-1)},0]$.
	For the corresponding maps $\fScaled{\tau_k}$,
	we calculate
	\begin{align*}
		\| \nabla \dd \fScaled{\tau_k}(x,\tScaled)\|^2_{\tau_k} &= \tau^{-2}_k \|\nabla \dd f(x,t)\|^2 = \frac{t_k}{\sigma_k^{1/(l-1)}} \|\nabla \dd f(x,t)\|^2 \\
			&\le \frac{2t}{\sigma^{1/(l-1)}_k} \|\nabla \dd f(x,t)\|^2 \stackrel{\text{Eq.\ \eqref{eq:decay2}}}{\le} \frac{2C_{1,\delta}}{\sigma^{1/(l-1)}_k} \,.
	\end{align*}
	Since $\sigma_k\to\infty$ for $k\to\infty$ and $l\ge 3$, we deduce
	\[
		\|\nabla \dd \fScaled{\tau_k}(\xScaled,\tScaled)\|^2_{\tau_k} \xrightarrow{k\to\infty} 0 \,.
	\]
	Now fix any $\eta$ with 
	$\eta\in\bigl( 0, \sigma_k^{1/(l-1)}\bigr)$ for all $k$.
	By lemma \ref{lem:HOEst1}, all higher order derivatives of $\fScaled{\tau_k}$
	are uniformly bounded on $\Rb^m\times[-\eta,0]$.
	Thus, $\fScaled{\tau_k}$ sub-converges on compact sets to a smooth solution of the quasilinear
	parabolic equation
	\[
		\partial_t \fScaled{\infty} = \sum_{i,j=1}^m \tgind_{\infty}^{ij} \frac{\partial^2 \fScaled{\infty}}{\partial x^i \partial x^j}
	\]
	on $\Rb^m\times[-\eta,0]$. But then $\|\nabla \dd \fScaled{\infty}\|_{\infty}=0$
	on $\Rb^m\times[-\eta,0]$ contradicts Eq.\ \eqref{eq:xkDef2},
	which implies $\|\nabla^{l-1} \dd \fScaled{\tau_k}(x_k,0)\|^2_{\tau_k}\ge 1$ for all
	$l\ge 3$ and any $k$.
\end{proof}

%
%
\section{Examples}

\begin{example}[Hyperbolic space I]
	\label{ex:HS1}
	We give an example of the mean curvature flow of
	a strictly length-decreasing map $f$,
	where the corresponding graphs $\Gamma(f_t)$ diverge to spatial infinity
	as time tends to infinity.
	\begin{enumerate}[label=(\roman*)]
		\item \textbf{Geometic Setup.} 
			Let $m=1$ and let $N=\Hb^2$ be the $2$-dimensional hyperbolic space.
			In the upper half-plane model, $\Hb^2$ is identified with
			\[
				\Hc \coloneqq \{ (y^1,y^2) \in \Rb^2 : y^2>0 \}
			\]
			with the metric
			\[
				\gind_{\Hc} \coloneqq \frac{1}{(y^2)^2}\bigl( \der y^1\otimes \der y^1 + \der y^2 \otimes \der y^2 \bigr) \,.
			\]
			The only non-vanishing Christoffel symbols are given by
			\[
				\Gamma_{12}^1 = \Gamma_{21}^1 = -\Gamma_{11}^2 = \Gamma_{22}^2 = -\frac{1}{y^2} \,.
			\]
			The induced metric of the graph at the point $(x,f(x))$ is given by
			\[
				\gind = \left( 1 + \frac{(\partial_x f^1(x))^2 + (\partial_x f^2(x))^2}{(f^2(x))^2} \right) \der x \otimes \der x
			\]
			with inverse
			\[
				\gind^{-1} = \frac{(f^2(x))^2}{(f^2(x))^2+(\partial_x f^1(x))^2 + (\partial_x f^2(x))^2} \frac{\partial}{\partial x} \otimes \frac{\partial}{\partial x} \,.
			\]
			Then equation \eqref{eq:pMCF} reads
			\begin{align}
				\label{eq:MCFex1}
				\partial_t f^1(x) &= \gind^{-1} \left( \partial_x^2 f^1(x) - 2 \frac{1}{f^2(x)} (\partial_x f^1(x)) (\partial_x f^2(x)) \right) \,, \\
				\label{eq:MCFex2}
				\partial_t f^2(x) &= \gind^{-1} \left( \partial_x^2 f^2(x) + \frac{1}{f^2(x)} \bigl( (\partial_x f^1(x))^2 - (\partial_x f^2(x))^2 \bigr) \right) \,.
			\end{align}
			In these coordinates, $f$ is (strictly) length-decreasing at $x\in\Rb$ if
			\[
				\frac{1}{(f^2(x))^2}\Bigl( (\partial_xf^1(x))^2+(\partial_xf^2(x))^2 \Bigr) \le 1-\delta
			\]
			for some $\delta\in(0,1]$.
		\item \textbf{Solution to the Mean Curvature Flow.}
			Let us make the ansatz 
			\[
				f_t(x) = \bigl( f^1_t(x), f^2_t(x) \bigr) = \bigl( x, d(t) \bigr)
			\]
			for some function $d:\Rb \to \Rb$ which is to be determined. 
			Since $\partial_x f^1_t(x) = 1$ and $\partial_x f^2_t(x)=0$, the length-decreasing
			condition is equivalent to
			\[
				\frac{1}{d^2(t)} \le 1-\delta \qquad \Leftrightarrow \qquad d^2(t) \ge \frac{1}{1-\delta} \,.
			\]
			Further, equation \eqref{eq:MCFex1} is identically satisfied, while
			equation \eqref{eq:MCFex2} evaluates to
			\[
				\partial_t d(t) = \frac{d(t)}{d^2(t)+1} \,.
			\]
			Integrating this equation, we obtain
			\[
				t-t_0 = \ln \left[ d(t) \exp \left( \frac{1}{2} d^2(t) \right) \right] = \frac{1}{2} \ln \left[ d^2(t) \exp\bigl( d^2(t) \bigr) \right] \,.
			\]
			The (increasing) solution for $d(t)$ is then given by
			\[
				d(t) = \sqrt{\mathrm{W}\bigl( \exp\bigl(2(t-t_0)\bigr) \bigr)} \,,
			\]
			where $\mathrm{W}$ denotes the principal branch of the Lambert W function.
			In particular,
			the lower bound from the length-decreasing condition is preserved, and in
			the limit $t\to\infty$, we have $f_t^*\gind_{\Hc} \to 0$. 
			Using the asymptotic expansion 
			\[
				\sqrt{\mathrm{W}\bigl(\exp\bigl(2(t-t_0)\bigr)\bigr)} \quad \stackrel{t\to\infty}{\approx} \quad \sqrt{\ln \frac{\exp\bigl(2(t-t_0)\bigr)}{2(t-t_0)}} \,,
			\]
			we also see that $d(t)\to\infty$ as $t\to\infty$. 
			
			The embedding corresponding
			to the graph of $f_t$ is given by
			\[
				F_t(x) = \bigl( x, x, d(t) \bigr) = \left( x, x, \sqrt{\mathrm{W}\bigl( \exp\bigl(2(t-t_0)\bigr) \bigr)}  \right) \in \Rb \times \Hc \,.
			\]
			Since the induced metric on the graph
			\[
				\gind(\partial_x,\partial_x) = 1 + \frac{1}{d^2(t)} = 1 + \frac{1}{\mathrm{W}\bigl(\exp\bigl(2(t-t_0)\bigr)\bigr)}
			\]
			does not depend on $x$, the Christoffel symbols of the graph vanish and
			the second fundamental form evaluates to
			\[
				\A(\partial_x,\partial_x) = \nabla^{\gind_{\Rb\times\Hc}}_{\dd F(\partial_x)}\dd F(\partial_x) = \frac{1}{d(t)} \frac{\partial}{\partial y^2} \,.
			\]
			Consequently, the mean curvature vector is given by
			\[
				\Hv = \gind^{-1} \A(\partial_x,\partial_x) = \frac{d^2(t)}{1+d^2(t)} \frac{1}{d(t)} \frac{\partial}{\partial y^2} = \frac{d(t)}{1+d^2(t)} \frac{\partial}{\partial y^2} \,.
			\]
			From the differential equation satisfied by $d$ we also obtain
			\[
				\partial_t F_t(x) = \frac{d(t)}{1+d^2(t)} \frac{\partial}{\partial y^2} \,,
			\]
			which shows that $F_t$ is indeed a solution of the mean curvature flow.
		\item\textbf{Decay Behavior.}
			Let us calculate the square norm of $\Hv$,
			\[
				\|\Hv\|^2 = \frac{1}{(f^2(x))^2} \frac{d^2(t)}{\bigl(1+d^2(t)\bigr)^2} = \frac{1}{d^2(t)} \frac{d^2(t)}{\bigl(1+d^2(t)\bigr)^2} = \frac{1}{\bigl(1+d^2(t)\bigr)^2} \,.
			\]
			In particular, since
			\[
				\|\Hv\| \approx \frac{1}{2(t-t_0) - \ln(2(t-t_0))}
			\]
			for large $t$, we have
			\[
				c_1 \le t\|\Hv\| \le c_2
			\]
			for some $0<c_1<c_2$ and large $t$. Further, for sufficiently large
			$t_1<t_2$, we calculate the length
			of the curve $\gamma:[t_1,t_2]\to \Rb\times \Hc$ given by
			\[
				\gamma(x,t) \coloneqq F_t(x)
			\]
			to be
			\begin{align*}
				L(\gamma) &= \int_{t_1}^{t_2} \| \partial_t F_t(x) \| \der t = \int_{t_1}^{t_2} \bigl\|\Hv(x,t)\bigr\| \der t \\
					& \ge \int_{t_1}^{t_2} \frac{c_1}{t} \der t = c_1 \bigl( \ln t_2 - \ln t_1 \bigr) \,.
			\end{align*}
			This diverges as $t_2\to\infty$ and agrees with the previous discussion.
	\end{enumerate}
\end{example}

\begin{example}[Hyperbolic space II]
	\label{ex:HS2}
	In this example, we construct an explicit solution to the graphical mean curvature
	flow of a map $f:\Rb\to\Hb^2$, where this time we use the Poincar\'{e} disk
	as a model for the hyperbolic space.
	\begin{enumerate}[label=(\roman*)]
		\item \textbf{Geometric Setup.} 
			The Poincar\'{e} disk model of hyperbolic space is the open unit disk $D\subset\Rb^2$
			endowed with the metric
			\[
				\gind_D \coloneqq \frac{4}{(1-r^2)^2} \bigl( \der x \otimes \der x + \der y \otimes \der y \bigr) \,,
			\]
			where $r^2\coloneqq x^2+y^2$. The inverse metic can be written as
			\[
				\gind_D^{-1} = \frac{(1-r^2)^2}{4}\bigl( \partial_x\otimes\partial_x + \partial_y\otimes\partial_y\bigr)
			\]
			and for the Christoffel symbols, we obtain
			\begin{align*}
				\Gamma_{11}^1 &= \Gamma_{12}^2 = \Gamma_{21}^2 = -\Gamma_{22}^1 = \frac{2x}{1-r^2}\,, \\
				\Gamma_{12}^1 &= \Gamma_{21}^1 = -\Gamma_{11}^2 = \Gamma_{22}^2 = \frac{2y}{1-r^2} \,.
			\end{align*}
			The metric induced on the graph of a function $f:\Rb\to D$ is given by
			\[
				\gind = \left( 1 + 4 \frac{(\partial_x f^1)^2 + (\partial_x f^2)^2}{\left( 1 - (f^1)^2 - (f^2)^2 \right)^2} \right) \der x \otimes \der x \,.
			\]
			The mean curvature flow system \eqref{eq:pMCF} reads
			\begin{align*}
				\partial_t f^1 &= \gind^{-1} \biggl( \partial_x^2 f^1 + \frac{2}{1-(f^1)^2 - (f^2)^2} \Bigl( f^1 (\partial_x f^1)^2 - f^1 (\partial_x f^2)^2 \\
					& \qquad \qquad \qquad \qquad \qquad \qquad \qquad \qquad \qquad \qquad + 2 f^2 (\partial_x f^1)(\partial_x f^2) \Bigr) \biggr) \,, \\
				\partial_t f^2 &= \gind^{-1} \biggl( \partial_x^2 f^2 + \frac{2}{1- (f^1)^2 - (f^2)^2} \Bigl( -f^2(\partial_x f^1)^2 + f^2(\partial_x f^2)^2 \\
					& \qquad \qquad \qquad \qquad \qquad \qquad \qquad \qquad \qquad \qquad + 2 f^1 (\partial_x f^1)(\partial_x f^2) \Bigr) \biggr) \,.
			\end{align*}
			The map $f$ is strictly length-decreasing precisely if
			\[
				f^*\gind_D(v,v) = 4 \frac{(\partial_x f^1)^2 + (\partial_x f^2)^2}{\left( 1 - (f^1)^2 - (f^2)^2 \right)^2} \|v\|_{\Rb}^2 \le (1-\delta) \|v\|_{\Rb}^2
			\]
			for some $\delta\in(0,1]$ and all vector fields $v$.
		\item \textbf{Solution to the Mean Curvature Flow.}
			Let us make the ansatz
			\[
				f_t (x) = \bigl( r(t) \sin(x) , r(t) \cos(x) \bigr) \,.
			\]
			Then $f_t$ is strictly length-decreasing precisely if
			\[
				4 \frac{r^2}{\bigl(1-r^2\bigr)^2} \le 1-\delta \qquad \Leftrightarrow \qquad r \le \sqrt{2} - 1 - \widetilde{\delta}
			\]
			for some $\widetilde{\delta}\in(0,\sqrt{2}-1]$ (depending on $\delta$).
			The induced metric simplifies to
			\[
				\gind = \left( 1 + 4 \frac{r^2}{(1-r^2)^2} \right) \der x \otimes \der x = \left( \frac{1+r^2}{1-r^2} \right)^2 \der x \otimes \der x
			\]
			and its inverse is given by
			\[
				\gind^{-1} = \left( \frac{1-r^2}{1+r^2} \right)^2 \partial_x \otimes \partial_x \,.
			\]
			The mean curvature flow system reads
			\[
				\partial_t r(t) = - r(t) \frac{1-r^2(t)}{1+r^2(t)}
			\]
			and the solution for $r(t)\le 1$ to this ordinary differential equation is given by
			\[
				r(t) = \frac{1}{2} \Bigl( \sqrt{ c_1^2 \ee^{2t} + 4} - c_1 \ee^t \Bigr) \,, \qquad c_1\ge 0 \,.
			\]
			We remark that
			\[
				r(t) = \frac{1}{\sqrt{ c_1^2 \ee^{2t} + 4} +  c_1 \ee^t} \qquad \xrightarrow{t\to\infty} \qquad 0 \,.
			\]
		\item \textbf{Decay Behavior.}
			Since $\gind$ only depends on $r$, which does not depend on $x$, 
			the Christoffel symbols of the induced metric vanish, and the second fundamental
			form is given by
			\[
				\A_{xx} = - r(t)\sin(x) \frac{1+r^2(t)}{1-r^2(t)} \frac{\partial}{\partial y^1} - r(t) \cos(x) \frac{1+r^2(t)}{1-r^2(t)} \frac{\partial}{\partial y^2} \,.
			\]
			The mean curvature vector thus is
			\[
				\Hv = \gind^{-1} \A_{xx} = - r(t) \sin(x) \frac{1-r^2(t)}{1+r^2(t)} \frac{\partial}{\partial y^1} - r(t) \cos(x) \frac{1-r^2(t)}{1+r^2(t)} \frac{\partial}{\partial y^2} \,.
			\]
			Since
			\[
				\partial_t F_t(x) = (\partial_t r(t)) \sin(x) \frac{\partial}{\partial y^1} + (\partial_t r(t)) \cos(x) \frac{\partial}{\partial y^2} \,,
			\]
			the differential equation for the mean curvature flow reduces to the ordinary
			differential equation considered above.
			Further, 
			since
			\[
				r(t) \le \frac{1}{2 c_1} \ee^{-t} \,,
			\]
			for the square norm of the mean curvature vector we obtain
			\[
				\|\Hv\|^2 = \frac{4 r^2(t)}{(1+r^2(t))^2} \le \frac{1}{c_1^2} \ee^{-2t} \,.
			\]	
	\end{enumerate}
\end{example}

\begin{example}[Hyperbolic Space III]
	\label{ex:HS3}
	We give explicit examples of strictly length-decreasing maps $f:\Rb\to\Hb^2$ which are stationary points
	of the mean curvature flow.
	\begin{enumerate}[label=(\roman*)]
		\item In the upper half-plane model $\Hc$, consider the family of functions given by
			\[
				f:\Rb\times[0,T)\to\Hc \,, \qquad f(x,t) \coloneqq \bigl( x_0, \exp(cx) \bigr) \,,
			\]
			where $0\le c\le 1-\delta$ is fixed. Then for any $v\in \Gamma(T\Rb)$ we have
			\[
				f^*\gind_{\Hc}(v,v) = c^2 \|v\|_{\Rb}^2 \le (1-\delta) \|v\|_{\Rb}^2 \,,
			\]
			so that $f$ is a (family of) strictly length-decreasing maps. Further, since
			\begin{multline*}
				\partial_x^2 f^2_t(x) + \frac{1}{f^2_t(x)}\Bigl( (\partial_x f^1_t(x)^2-(\partial_x f^2_t(x)^2\Bigr) \\
					= c^2 \exp(cx) + \frac{1}{\exp(cx)}\bigl( -c^2 (\exp(cx))^2 \bigr) = 0 \,,
			\end{multline*}
			which mean that $f_t(x)$ is stationary under graphical mean curvature flow.
		\item In the disk model, consider the family of functions given by
			\[
				f:\Rb\times[0,T)\to D\,, \qquad f(x,t) \coloneqq \left( \tanh\left(\frac{c}{2}x\right), 0 \right) \,,
			\]
			where $0\le c\le 1-\delta$ is fixed. Then for any $v\in\Gamma(T\Rb)$ we have
			\[
				f^*\gind_D(v,v) = c^2 \|v\|_{\Rb}^2 \le (1-\delta) \|v\|_{\Rb}^2 \,,
			\]
			so that $f$ is a (family of) strictly length-decreasing maps. Using the ansatz above, the equation for the
			first component of $f_t(x)$ that needs to be satisfied for the graphical mean curvature
			flow is given by
			\begin{align*}
				\partial_t f^1(x,t) &= \gind^{-1} \partial_x^2 f^1(x,t) + \gind^{-1} \sum_{a,b=1}^2 \Gamma_{ab}^1 (\partial_x f^a(x,t))(\partial_x f^b(x,t)) \\
					&= \gind^{-1}\left( \partial_x^2 f^1(x,t) + \frac{2 f^1(x,t)}{1-(f^1(x,t))^2} (\partial_x f^1(x,t))^2 \right) \,.
			\end{align*}
			Since
			\begin{align*}
				\partial_x^2 f^1(x,t) &= - \frac{c^2}{2} \tanh\left( \frac{c}{2} x \right) \left( 1 - \tanh^2 \left( \frac{c}{2} x \right) \right) \\
					&= - c \tanh\left( \frac{c}{2} x \right) \partial_x f^1(x,t) = -c f^1(x,t) \partial_x f^1(x,t) \\
					&= - \frac{2}{1-(f^1(x,t))^2}(\partial_x f(x,t))^2 \,,
			\end{align*}
			the mean curvature flow equation is identically satisfied.
	\end{enumerate}
\end{example}

\vskip 5\baselineskip

\begin{flushleft}	
	\textsc{Felix Lubbe} \\
	\textsc{Department of Mathematics} \\
	\textsc{University of Hamburg} \\
	\textsc{Bundesstr.\ 55}\\
	\textsc{D--20146 Hamburg, Germany} \\
	\textsl{E-mail address:} \textbf{\href{mailto:Felix.Lubbe@uni-hamburg.de}{Felix.Lubbe@uni-hamburg.de}}
\end{flushleft}

\end{document}